\documentclass[12pt,leqno]{amsart} 

\def\cal{\mathcal}

\def\int{\mathbb{Z}}

\def\Ueg{{\cal U}_{\varepsilon}(G)}

\def\g{\mathfrak g}
\def\Inn{\text{Inn}~}
\def\Ker{\text{Ker}~}
\def\Oc{{\cal O}}
 
\def\ll{{\mathfrak l}}
\def\h{{\mathfrak h}}

\def\p{{\mathfrak p}}
\def\pp{{\mathfrak p}}

\def\id{{\rm id}}

\def\spec{{\rm Spec}}

\def\ve{\varepsilon}
\def\ZZ{\mathbb{Z}}

\def\CC{\mathbb{C}}

\def\Hom{\operatorname{Hom}}
\def\Ind{\operatorname{Ind}}
\def\calO{\mathcal{O}}
\def\calB{\mathcal{B}}

\def\rank{\operatorname{rank}}
\def\Ann{\operatorname{Ann}}
\def\Lie{\operatorname{Lie}}
\def\Rep{\operatorname{Rep}}
\def\Spec{\operatorname{Spec}}
\def\id{\operatorname{id}}
\def\la{\langle}
\def\ra{\rangle}

\usepackage{amssymb}
\usepackage{amscd}
\usepackage{amsmath}
\usepackage[bookmarks]{hyperref}
\usepackage{a4wide}
\usepackage{tikz}
\usetikzlibrary{matrix}
\usepackage{enumerate}
\input xy 
\xyoption{all}
\title{On small modules for quantum groups at roots of unity}
\newcommand{\comgio}[1]{\textcolor{red}{\textbf{#1}}}
\newcommand{\comiul}[1]{\textcolor{blue}{\textbf{#1}}}

\newtheorem{theorem}{Theorem}[section]
\newtheorem{conjecture}{Conjecture}
\newtheorem{lemma}[theorem]{Lemma}
\newtheorem{corollary}[theorem]{Corollary}
\newtheorem{proposition}[theorem]{Proposition}

\theoremstyle{definition}
\newtheorem{remark}[theorem]{Remark}

\author{Giovanna Carnovale}%\\
\address{Dipartimento di Matematica\\
via Trieste 63 - 35121 Padova - Italy}
\email{carnoval@math.unipd.it}
\thanks{G.C. was partially supported by Grant CPDA125818/12 of Universit\`a di Padova and MIUR PRIN 2012KNL88Y\_004}

\author{Iulian I. Simion}%\\
\address{Dipartimento di Matematica\\
via Trieste 63 - 35121 Padova - Italy}
\email{iuliansimion@gmail.com}
\thanks{I.I.S. was supported by the Universit\`a di Padova (grants CPDR131579/13 and CPDA125818/12). }

\begin{document}

\maketitle

\begin{abstract}
A conjecture of De Concini Kac and Procesi provides a bound on the minimal possible dimension 
of an irreducible module for  quantized enveloping algebras at an odd root of unity. We pose the problem of the existence of modules whose dimension equals this bound.
We show that this question cannot have a positive answer in full generality and discuss variants of this question. 
%To this end, we relate Sevostyanov's theory with the representation theory of different isogeny classes of 
%quantized enveloping algebra at roots of unity.
\end{abstract}

%------------------------------------------------------------------------
%
% INTRO
%
%------------------------------------------------------------------------
\section{Introduction}

Let $\ell$ be an odd positive integer. A construction due to De Concini Kac and Procesi \cite{DCKP} associates with an
 irreducible representation $V$ of a quantized enveloping algebra at an
$\ell$-th root of unity a conjugacy class $\Oc_V$ in a corresponding algebraic group $G$. They conjectured
that the dimension of each irreducible module $V$ is divisible by $\ell^{\frac{1}{2}\dim \Oc_V}$.
The De Concini-Kac-Procesi (DCKP for short) conjecture was confirmed in \cite{DCKP-remarkable} for regular conjugacy classes, in \cite{cantarini-subregular} for subregular unipotent classes in type $A_n$, in \cite{cantarini-modp} for all conjugacy classes in type $A_n$  when $\ell$ is a prime, in \cite{cantarini-spherical} and \cite{ccc} for spherical conjugacy classes for arbitrary $G$. An approach to settle the conjecture for unipotent conjugacy classes was given in \cite{kremnitzer}. 
In \cite{sev1,sev2}, which at the time of writing are under refereeing process, Sevostyanov proposes a strategy for the proof of the DCKP conjecture for simply-connected quantized enveloping algebras, with some restrictions on $\ell$.  

We are interested in the assumptions under which the bound in the DCKP conjecture is attained, i.e.,
in conditions for the existence of a module $V$ of dimension $\ell^{\frac{1}{2}\dim \Oc_V}$. We call such a module $V$ a 
{\em small module}. In the case of modular Lie algebras, the analogue problem  was formulated in \cite[\S 8]{humphreys-conj} and 
is usually referred to as Humphreys conjecture. For an account on the development of this conjecture, see \cite{pre-man}.
%we will refer to this question, and its possible variants, as to the {\em quantum Humphreys conjecture}. 

%This question does not have a positive answer in full generality as is shown in Section \ref{sec:l_prime}. %for the algebras studied by Sevostyanov, cf. Remark \ref{rem:counter}. 
One of the big differences with the modular case is that for each reductive Lie algebra $\g$, there are several quantized enveloping algebras, each corresponding to a lattice $M$
 between the root lattice and the
weight lattice of $\g$. We say that $M$ is the isogeny type of the algebra. The theory is well developed for the simply connected case (i.e., when $M$ is the weight lattice),
 but does not seem to allow for standard 
inductive arguments as in the case of modular Lie algebras, as quantized Levi subalgebras are not always simply connected.

This motivates our analysis of different isogeny types both with respect to the DCKP-conjecture and to the question of the existence of small modules.

%For this reason we analyze here the theory for different isogeny classes and state variants of the quantum Humphreys conjecture in this setting. 

The paper is structured as follows: we recall in Section \ref{sec:preliminaries} the DCKP construction, in particular we see that every irreducible module for a quantized enveloping algebra is
a representation of some finite-dimensional quotient $U_\eta^M(\g)$, corresponding to an $\ell$-character $\eta$. Each $\ell$-character is associated with a conjugacy class in $G$. 
In Section \ref{sec:reduced} we relate the finite-dimensional quotients $U_\eta^M(\g)$ and $U_\eta^N(\g)$ corresponding to different isogeny types in Theorem \ref{thm:iso}: in particular we
provide necessary and sufficient conditions on $\ell$ ensuring that $U_\eta^M(\g)$ is independent of the isogeny type.  As a consequence, under these assumptions on $\ell$, in order to confirm  the DCKP conjecture for all isogeny types it is enough to prove it for one type (Corollary \ref{cor:l_prime}). In Section \ref{sec:positive} we give an inductive argument using parabolic induction
which reduces the quest for small modules  
to rigid conjugacy classes and settles their existence for $\mathfrak{sl}_{n+1}$ when $(\ell, (n+1)!)=1$. The results we get are in analogy with the ones in \cite{fri-par}. The last Section is devoted to 
the cases in which the assumption on $\ell$ is not verified. We show that in this case small modules may fail to exist for  $U_\eta^N(\g)$ even if they exist for $U_\eta^M(\g)$. 
A complete answer to the question of existence of small modules for central elements in $G$ is given in Proposition \ref{prop:central}. As a conclusion, we formulate Conjecture \ref{conj:small} which is 
a quantum analogue of Humphreys conjecture.

\section{Notation}
\label{sec:notation}

Let $\ve$ be a primitive $\ell$-th root of $1$. Throughout $\g$ denotes a fixed semisimple Lie algebra of rank $n$ with root lattice $Q$ and weight lattice $\Lambda$. 
We assume that $\ell$ is odd and coprime with $3$ if $\g$ has a component of type $G_2$.

For a lattice $M$ such that $Q\subseteq M\subseteq \Lambda$ 
we denote by $G_M$ the complex semisimple algebraic group with $\Lie(G_M)=\g$ and of isogeny type determined by $M$. We let $T_M\subseteq G_M$ be a 
maximal torus and $B_M\subseteq G_M$ a Borel subgroup containing $T_M$. 
The opposite Borel subgroup is denoted by $B_{M}^{-}$. The unipotent radical of $B_M$ does not depend on $M$ and we denote it by $U$ so $B_M=T_M\ltimes U$.
%and it is clear from the context that $U$ is viewed as a subgroup of $G_M$. 
Similarly, the radical of $B_M^{-}$ is denoted by $U^{-}$.

The set of roots w.r.t. $T_M$ is denoted by $\Phi$, the set of positive roots w.r.t. $B_M$ is denoted by $\Phi^{+}$ and $\Delta=\{\alpha_1,\ldots,\alpha_n\}$ stands for the set of simple roots determined by 
$B_{M}$. We use the Bourbaki numbering of the simple roots. If $\Pi\subseteq\Delta$ then $\Phi_{\Pi}:=\Phi\cap\ZZ\Pi$. The $i$th fundamental weight  is denoted by $\lambda_i$
 and the longest element of the Weyl group $W$ of $\Phi$ is denoted by $w_0$.

We choose the $W$-invariant Euclidean norm $(\_|\_)$ on ${\mathbb R}\Delta$ so that $(\alpha|\alpha)=2$ for short roots in all simple factors of $\g$. 

%the corrsponding Cartan matrix is $a_{ij}=2(\alpha_i,\alpha_j)/(\alpha_i,\alpha_i)$. In particular $2(\lambda_i,\alpha_j)/(\alpha_j,\alpha_j)=\delta_{ij}$ and 
%$\alpha_i=\sum_j a_{ji}\lambda_j$. There are relatively prime positive integers $d_1,\dots,d_n$ such that $(d_ia_{ij})$ is symmetric and $d_ia_{ii}=2$ for all short roots $\alpha_i$. This defines a perfect bilinear pairing
%$(\_|\_):\Lambda\times Q\rightarrow\ZZ$ given by $(\lambda_i|\alpha_j)=d_j\delta_{ij}$ such that $(\alpha_i|\alpha_j)=a_{ji}d_j=d_ia_{ij}$. \comgio{We assume throughout the paper that $(\ell, d_i)=1$ for every $i$, i.e., that $\ell$ is odd and
%coprime with $3$ if $\Phi$ has a component of type $G_2$.}%
%\comiul{This assumption is not visible here. Perhaps we could merge the intro section and the notation section by replacing the `2. Notation' title with a `bigskip' and after the notation is finished we put another `bigskip' and mention the assumptions on $\ell$ and the main results and goals. If we do this than we have the notation at our disposal. Also, it would be good to move the conditions for Humphreys conjecture in the intro so that the constant $b(\g)$ is available when we give the conditions for Sevostyanov's proof}

For the lattice $M$ as above, we denote by $U_\ve^M(\g)$ the De Concini-Kac quantized enveloping algebra at the root $\ve$ of one.
It is defined by generators $E_{\alpha}, F_{\alpha} (\alpha\in \Delta)$ and $K_{\mu}(\mu\in M)$ subject to the relations
\begin{eqnarray}
&
K_0=1,
\quad
K_{\mu}K_{\nu}=K_{\mu+\nu}
%\quad\text{for all }\mu,\nu\in M
\\
&
K_{\mu}E_{\alpha}=\ve^{(\mu\mid\alpha)}E_{\alpha}K_{\mu}
%\quad\text{for all }\mu\in M\text{ and }\alpha\in\Delta
\\
&
K_{\mu}F_{\alpha}=\ve^{-(\mu\mid\alpha)}F_{\alpha}K_{\mu}
%\quad(\mu\in\Pi_{M})
\\
&E_{\alpha}F_{\beta}-F_{\beta}E_{\alpha}=\delta_{\alpha,\beta}\frac{K_{\alpha}-K_{-\alpha}}{\ve_{\alpha}-\ve_{\alpha}^{-1}}
\\
&
\sum_{{a,b\geq0}
\atop
{a+b=1-2\frac{(\alpha|\beta)}{(\alpha|\alpha)}}}(-1)^{a}{a+b\brack a}_{\ve_\alpha}E_{\alpha}^{a}E_{\beta}E_{\alpha}^{b}=0
\\
&
\sum_{{a,b\geq0}
\atop
{a+b=1-2\frac{(\alpha|\beta)}{(\alpha|\alpha)}}}
(-1)^{a}{a+b\brack a}_{\ve_\alpha}F_{\alpha}^{a}F_{\beta}F_{\alpha}^{b}=0
\end{eqnarray}
%where we use divided power notation as in \cite{lusztig-rootof1} and
where ${c\brack d}_{e}=\frac{[c]_e!}{[d]_e![c-d]_e!}$, $[c]_e!=[c]_e[c-1]_e\cdots [1]_e$, $[c]_e=\frac{e^c-e^{-c}}{e-e^{-1}}$
and $\ve_\alpha=\ve^{\frac{(\alpha|\alpha)}{2}}$.

If $\g=\g_1\oplus\cdots \oplus\g_r$ is the decomposition of $\g$ into simple factors, 
 $M=M_1\oplus\cdots \oplus M_r$,
% we set $\ve_i:=\ve_\alpha$ if $(\alpha|\alpha)>2$ for all $\alpha\in\Phi_i$ and $\ve_i:=\ve$ otherwise.
we have the isomorphism
%\begin{equation}
%\label{eq:product}
$U_\ve^M(\g)\simeq \bigotimes_{i=1}^r U_{\ve}^{M_i}(\g_i)$. 
%\end{equation}

%$\Phi_{\Pi}=\Phi_1\sqcup \cdots \sqcup \Phi_r$ and
%$Q_{\Pi}=Q_{\Pi_1}\oplus\cdots \oplus Q_{\Pi_r}$,

%The algebra  $U_\ve^M(\g)$ is ${\mathbb Z}\Delta$-graded, by setting $\deg(E_\alpha)=\alpha$, $\deg(F_\alpha)=-\alpha$, $\deg(K_\lambda)=0$ \comiul{[we don't use this grading]}.
For every choice of the lattice $M$ and every reduced decomposition $w_0$, we have root vectors $E_\gamma$, $F_\gamma$ for $\gamma\in\Phi^+$ \cite{lusztig-rootof1}.
The set of root vectors of $U_\ve^M(\g)$ is the union of the sets of root vectors for each factor $U_\ve^{M_i}(\g_i)$ . 
%ir construction is compatible with the product decomposition.
%which are homogenous of degree $\gamma$ and $-\gamma$, respectively 
%\cite{lusztig-rootof1}. \comgio{Even if the algebra defined in there is slightly different, the construction of the root vectors takes place in $U_\ve^Q(\g)$ and does not depend on the chosen lattice.}

%For an $r$-tuple of primitive $\ell$-th roots of unity 
%$\boldsymbol{\epsilon}=(\varepsilon_1,\ldots,\varepsilon_r)$ we write
%$U_{\boldsymbol{\epsilon}}^M(\g)$ for the product $\bigotimes_{i=1}^r U_{\ve_i}^{M_i}(\g_i)$. 

%------------------------------------------------------------------------
%
% NOTATION
%
%------------------------------------------------------------------------
%\section{Notation}\label{sec:notation}
%\begin{multicols}{2}
\medskip

The notation for the main objects that we consider is as follows: 
\begin{itemize}%[label={}]
%\item $\g$ a simple complex Lie algebra
%\item $Q\subseteq \Lambda$ the root lattice and the weight lattice of $\g$.
%\item $G_{M}$ the algebraic group of type $M$ for a lattice $Q\subseteq M\subseteq\Lambda$
%\item $T_{M}\subseteq B_{M}$ a fixed maximal torus and a fixed Borel subgroup of $G_{M}$
%\item $\ell$ odd integer ( such that $(d_i,\ell)=1$ where $d_i$ \dots )
%\item $\ve$ primitive $\ell$-th root of unity
%\item $g_s,g_u$ the semisimple and unipotent part of $g\in G_M$ in the Jordan decomposition
\item $\calO_g^{H}$ denotes the conjugacy class of $g$ in the group $H$, if it is clear what $H$ is we write $\calO_g$
\item $C_H(S)$ for $H$ a group and $S\subseteq H$ is the centralizer of $S$ in $H$
\item $H^\circ$ denotes the identity component of the algebraic group $H$
\item $\Rep(A)$ is the set of modules of the algebra $A$ up to isomorphism
\item $\Spec(A)$ is the set of simple modules of the algebra $A$ up to isomorphism
%\item $U_{\ve}^{M}(\g)$ the \emph{quantized enveloping algebra of type $M$ at $\ve$}
\item $U_{\ve}^{M}(\h)$ is the subalgebra of $U_{\ve}^{M}(\g)$ generated by the $K_\lambda$ for $\lambda\in M$
%\item $U_{\ve}^{X}$ for $X\subseteq M$ is the subalgebra of $U_{\ve}$ ...
\item $Z_{\ve}^{M}(\g)$ is the center of $U_{\ve}^{M}(\g)$
\item $Z_{0}^{M}(\g)$ is the \emph{$\ell$-center} of $U_{\ve}^{M}(\g)$, i.e. the central subalgebra generated by $E^\ell_\alpha$, $F^\ell_\alpha$, for $\alpha\in\Phi^+$ and $K_{\mu}^{\ell}$ for $\mu\in M$
\item $U_{\eta}^{M}(\g)=U_{\ve}^{M}(\g)/(\ker\eta)U_{\ve}^{M}(\g)$ is the \emph{reduced quantized enveloping algebra} corresponding to $\eta\in\Spec(Z_0^{M}(\g))$.
%\item $X^{Y}(Z)$ where $X\in\{U_{\ve},U_{\eta}\}$, $Y\subseteq M$ and $Z\subseteq \Phi$, is the subalgebra of $X^{M}$ generated by $\{K_{\mu}:\mu\in Y\}$ and root vectors corresponding to roots in $Z$
%\item $Z_0^{Y}(Z)$ where $Y\subseteq M$ and $Z\subseteq \Phi$, is the subalgebra of $X$ generated by $\{K_{\mu}^{\ell}:\mu\in Y\}$ and $\ell$-th powers of root vectors corresponding to roots in $Z$
%\item $X(Z)$ where $X\in\{U_{\ve}^{M},U_{\eta}^{M},Z_{0}^{M}\}$ and $Z\subseteq \Phi$ stands for $X(M,Z)$.
\item $U_{\eta}^{M}(\h)$ is the subalgebra of $U_{\eta}^{M}(\g)$ generated by $K_{\lambda}$ for $\lambda\in M$ 
%where $X\in\{U_{\ve}^{M},U_{\eta}^{M},Z_{0}^{M}\}$ stands for $X(M,\emptyset)$.
%$U_{\eta}^{M}(X)$ is short for $U_{\eta}^{M,M}(X)$.
%\item $Z_{0}^{M}(\g)$ the $\ell$-center of $U_{\ve}^{M}$
%\item $d_{\eta}$ half the dimension of the conjugacy class of $\eta\in G_{M}$ for some/any $M$.
\end{itemize}
%Column 2
%\end{multicols}
%Fix $\Pi\subseteq\Delta$ and let $\Phi_{\Pi}=\Phi\cap\ZZ\Pi$. The set $\Pi$ defines the parabolic subalgebra $\mathfrak{p}\supseteq\Lie(U)$ with Levi subalgebra $\mathfrak{l}$ then 
%$U_{\ve}^{M}(\mathfrak{p})=U_{\ve}^{M}(\Phi_{\Pi}\cup\Phi^{+})$, $U_{\eta}^{M}(\mathfrak{l})=U_{\ve}^{M}(\Phi_{\pi})$ and $Z_0^{Q}([\mathfrak{l},\mathfrak{l}])=Z_0^{Q_{\Pi}}(\Phi_{\Pi})$.

%The centralizer of $k$ in $K$ will be denoted by $K^k$, whereas the centralizer in $K$ of a subgroup $S$ will be denoted by $C_K(S)$. 
%For any algebraic group $K$, its identity component will be indicated by $K^\circ$.
%for $g\in G$ we write $g=g_sg_u$ for the Jordan \CC[T_N]decomposition of $g$ where $g_s$ is the semisimple part and $g_u$.
%By a gothic letter we will usually indicate the Lie algebra of the group indicated by the corresponding capital letter, for instance $\g={\rm Lie}(G)$, $\h={\rm Lie}(H)$. 
%The subalgebra generated by the $K_\lambda$ for $\lambda\in M$ will be denoted by $U_\ve^M(\h)$.
%The skew primitive generators, corresponding to the simple roots and their negatives, are denoted by $E_\alpha$ and $F_\alpha$, for $\alpha\in\Delta$. 

%and isogeny class of the group $G'$ in \cite[page 175]{DCKP}.

%------------------------------------------------------------------------
%
% THE SETTING
%
%------------------------------------------------------------------------
\section{Preliminaries}
\label{sec:preliminaries}

\subsection{}
\label{subsec:poisson}

For any reduced decomposition of $w_0$, we have an ordering of the positive roots and a choice of root vectors. Ordered monomials in the 
 root vectors and the $K_{\mu_j}$, for a ${\mathbb Z}$-basis $\mu_1,\,\ldots,\mu_n$ of $M$,
 form a PBW basis for $U_\ve^M(\g)$. This is proved in \cite[Proposition 4.2]{lusztig-rootof1} for $M=Q$ and the construction works for general $M$.
%We follow De Concini and Kac in associating to a lattice $M$ between $Q$ and $\Lambda$ a quantum group and, by construction all $U_{\ve}^{M}(\g)$ have a PBW-basis.

In addition, $U_\ve^M(\g)$ is a Hopf algebra, \cite[\S9.1]{DCP}. Its $\ell$-center is the central Hopf subalgebra $Z_0^M(\g)$ \cite[\S5.6 Proposition (d)]{DCKP} 
generated by $\ell$-th powers of the root vectors and of the $K_{\mu}$ \cite[\S4]{DCK}. Considering the same ordering of positive in order to describe $U$ as a product of root subgroups, the group $U$ is identified with
$\Spec(\CC[E_\alpha^\ell,\,\alpha\in\Phi^+])$; similarly,
 $U^-\simeq\Spec(\CC[F_\alpha^\ell,\,\alpha\in\Phi^+])$ and 
$\Spec({\mathbb C}[K_{\mu_j}^{\pm\ell},\,j=1,\,\ldots,\,n])\simeq {\rm Hom}_{\mathbb Z}(M, {\mathbb C}^\times)\simeq T_M$.
Combining these identifications we get isomorphisms of varieties 
\begin{eqnarray*}
\Spec(Z_0^M(\g))&\simeq &  \Spec(\CC[F_\alpha^\ell,\alpha\in\Phi^+])\times \Spec({\mathbb C}[K_{\mu_j}^{\pm\ell},j=1,\ldots,n])\times \Spec(\CC[E_\alpha^\ell,\alpha\in\Phi^+])\\
& \simeq & U^-\times T_M \times U\\
&\simeq & \{(t^{-1}u,tv)~|~t\in T_M, u\in U^-,v\in U\}.
\end{eqnarray*}
In addition, $Z_0^M(\g)$ is a finitely generated commutative Poisson Hopf algebra  \cite[\S7.3]{DCKP} and the composition of the above isomorphisms is an isomorphism of Poisson algebraic groups \cite[\S7.6]{DCKP}.

If $M\subset N$, the natural inclusion $\iota\colon U_{\ve}^{M}(\g)\subset U_{\ve}^{N}(\g)$ 
induces an inclusion  $Z_0^M(\g)\to Z_0^N(\g)$ which we denote by the same symbol. 

\subsection{}
\label{subsec:DCKPmap}
Considering the PBW basis, $U_{\ve}^{M}(\g)$ is a free $Z_0^M(\g)$-module 
with basis given by the  subset of the PBW monomials of $U_\ve^M(\g)$ with all exponents in $\{0,\,\ldots,\ell-1\}$. As a consequence, $U_{\ve}^{M}(\g)$ is finite over its center and 
every irreducible representation of $U_{\ve}^{M}(\g)$ is finite-dimensional \cite[\S11.1A]{CP}. By Schur's Lemma we have the natural restriction map $\xi_M\colon \Spec(U_{\ve}^{M}(\g))\to \Spec (Z_0^M(\g))$. 
Let $\pi_M\colon \Spec(Z_0^M(\g))\to B^-_MB_M$ be given by $(x,y)\mapsto x^{-1}y$, where we use the identification in \S \ref{subsec:poisson}. By the formulas in \cite[0.5]{DCKP} the map $\pi_M$ is the product of the 
$\pi_{M_i}$ for each simple component of $\g$, since $Z_0^M(\g)=\bigotimes_{i=1}^rZ_0^{M_i}(\g_i)$ (notation as in \S \ref{sec:notation}).

If $M\subset N$, we have an isogeny $\phi_{NM}\colon G_N\to G_M$.
Again using \S \ref{subsec:poisson}, $\iota^\ast$ is the restriction of $\phi_{NM}\times \phi_{NM}$ to $\{(t^{-1}u,tv)~|~t\in T_N, u\in U^-,v\in U\}\leq G_N\times G_N$.
%\comiul{[I don't understand this statement, what morhism, induced where]}. 
We have surjective horizontal maps in the commutative diagram below \cite{DCKP}.

\begin{figure}[!h]
\centering
\begin{tikzpicture}
  \matrix (m) [matrix of math nodes, row sep=3em, column sep=3em]{
    |[name=UM]| \Spec (U_{\ve}^{M}(\g)) & |[name=Z0M]|\Spec (Z_{0}^{M}(\g))  & |[name=GM]| B_{M}^{-}B_M\subseteq G_{M}\\
    |[name=UN]| \Spec (U_{\ve}^{N}(\g)) &  |[name=Z0N]|\Spec (Z_{0}^{N}(\g)) & |[name=GN]| B_{N}^{-}B_N\subseteq G_{N}\\
    };
  \draw[->]
        (UM) edge node[auto] {$\xi_M$}  (Z0M)
        (Z0M) edge node[auto] {$\pi_M$} (GM)

        (UN) edge node[auto] {$\xi_N$}  (Z0N)
        (Z0N) edge node[auto] {$\pi_N$} (GN)

        (Z0N) edge node[auto] {$\iota^{\ast}$} (Z0M)
        (GN) edge node[auto] {$\phi_{NM}$} (GM)
  ;
\end{tikzpicture}
\end{figure}
Let $\eta\in \Spec(Z_0^M(\g))$ and let  $V\in \Spec(U_{\ve}^M(\g))$. We say that $V$ has \emph{$\ell$-character} $\eta$ if $\xi_M(V)=\eta$. Since $\xi_M$ is surjective and $\iota^*$ is induced from 
$\phi_{NM}\times \phi_{NM}$,
%\comiul{[surjectivity of this map is not enough for the conclusion, we need $\iota^{\ast}$ is $\phi_{NM}$]}
any $\ell$-character of $U_{\ve}^{M}(\g)$ lifts to exactly 
$|N/M|$ $\ell$-characters of $U_{\ve}^{N}(\g)$. Moreover, considering the rows of the diagram, every $V\in\Spec(U_{\ve}^M(\g))$ is an irreducible module over the algebra
$$
U_{\eta}^M(\g):=U_{\ve}^M(\g)/({\rm Ker }\,\eta)U_{\ve}^M(\g)
$$
for $\eta=\xi_M(V)$. We call such quotients {\em reduced quantized enveloping algebra}. 

By abuse of notation, we denote the image of the generators of $U_{\ve}^M(\g)$ in $U_{\eta}^M(\g)$ with the same symbols.
% Using the  $Z_0^M(\g)$-basis of $U_{\ve}^{M}(\g)$ 
%with basis given by the  subset of the PBW basis elements of $U_\ve^M(\g)$ with all exponents in $\{0,\,\ldots,\ell-1\}$ 
%(This is clear from the PBW basis of $U_{\eta}^{M}(\g)$
%and the definition of $Z_{0}^{M}(\g)$). 
For any $\eta\in\Spec(Z_0^{M}(\g))$, the image of any  $Z_0^M(\g)$-basis of $U_{\ve}^{M}(\g)$  through the canonical projection $U_{\ve}^M(\g)\rightarrow U_{\eta}^M(\g)$ 
is a linear basis of $U_{\eta}^M(\g)$. In particular $\dim U_{\eta}^M(\g)=\ell^{\dim \g}$ and fixing a basis $X=\{\mu_1,\dots,\mu_n\}$ of $M$, we may choose the basis to consist of monomials of the form 
$$
{\bf F}^A{\bf K}_X^B{\bf E}^C:=F_{\alpha_N}^{a_N}\cdots F_{\alpha_1}^{a_1}K_{\mu_1}^{b_1}\dots K_{\mu_n}^{b_n}E_{\alpha_1}^{c_1}\cdots E_{\alpha_N}^{c_N}
$$
where $N=|\Phi^{+}|$, $B=(b_1,\dots,b_n)$, $A=(a_N,\dots,a_1)$ and $C=(c_1,\dots,c_N)$ are tuples of elements in $\{0,\,\ldots,\ell-1\}$.
\subsection{}
\label{subsec:automorphisms}
Let $(\sigma_1,\,\ldots,\,\sigma_n)\in\{\pm1\}^n$. For a given ${\mathbb Z}$-basis $\mu_1,\,\ldots,\mu_n$ of the lattice $M$, we consider the algebra automorphism $\sigma$ of $U_\ve^M(\h)$ 
given by  $\sigma(K_{\mu_i})=\sigma_i K_{\mu_i}$.  
For every $j=1,\,\ldots,\,n$ we have $\sigma(K_{\alpha_j})=\tau_j K_{\alpha_j}$  for some $\tau_j\in\{\pm1\}$. Then, $\sigma$ extends to an algebra automorphism of  $U_\ve^M(\g)$ by setting
$\sigma(F_{\alpha_j})=F_{\alpha_j}$ and $\sigma(E_{\alpha_j})=\tau_j  E_{\alpha_j}$.
% Indeed $\sigma$ preserves the relations defining the algebra $U_{\ve}^{M}(\g)$.

\subsection{}
For a lattice $M$, let $M'=\Lambda\cap \frac{1}{2}M$, and the corresponding central isogeny $\phi_{M'M}\colon G_{M'}\to G_M$. By \cite[p. 176]{DCKP}, the map $\pi_M$ 
factors through the big cell of $G_{M'}$, so we have the following composition of maps:

\begin{figure}[!h]
\centering
\begin{tikzpicture}
  \matrix (m) [matrix of math nodes, row sep=3em, column sep=3em]{
    |[name=UM]| \Spec (U_{\ve}^{M}(\g)) & |[name=Z0M]|\Spec (Z_{0}^{M}(\g))  & |[name=GM']| B_{M'}^{-}B_{M'}\subseteq G_{M'}   & |[name=GM]| B_{M}^{-}B_{M}\subseteq G_{M}  \\
      };
  \draw[->]
        (UM) edge node[auto] {$\xi_M$}  (Z0M)
        (Z0M) edge node[auto] {$\pi'$} (GM')
       (GM') edge node[auto] {$\phi_{M'M}$} (GM)
  ;
\end{tikzpicture}
\end{figure}

By \cite[\S 6.6 Theorem]{DCKP}, if $g=\pi'(\eta)\in G_{M'}\setminus Z(G_{M'})$, 
then for every $\tau \in (\pi')^{-1}(\calO_g)$ we have $U_{\eta}^M(\g)\simeq U_{\tau}^M(\g)$. 
Now, $\calO_g$ has a representative in $B_{M'}^-$, so we may assume $g=\pi'(\eta)\in B_{M'}^-$. Therefore $\pi_M(\eta)\in B_M^-$. Thus, by definition of the map $\pi_M$ in \S \ref{subsec:DCKPmap} and the identification in \S \ref{subsec:poisson}, we have $\eta(E_\alpha^{\ell})=0$ for every $\alpha\in\Phi^+$. If $h=\phi_{M',M}(g)\in B^-_M$ then $\pi_M^{-1}(\calO_h)\subset \bigcup_{z\in \Ker\phi_{M',M}} (\pi')^{-1}(\calO_{zg})$. 
For any $z\in \Ker \phi_{M',M} $, we choose $\delta\in (\pi')^{-1}(zg)$. Since $\Ker \phi_{M',M} \subset\{z\in Z(G_{M'})~|~ z^2=1\}$, we have $\eta(E_\alpha^\ell)=\delta(E_\alpha^\ell)=0$, $\eta(F_\alpha^\ell)=\delta(F_\alpha^\ell)$ 
and $\eta(K_{\mu_j}^\ell)=\sigma_j\delta(K_{\mu_j}^\ell)$, for $\sigma_j\in\{\pm1\}$. The automorphism $\sigma$ from paragraph \ref{subsec:automorphisms} induces an isomorphism $U_{\eta}^M(\g)\simeq U_{\delta}^M(\g)$. 
In other words, if $h=\pi_M(\eta)\in G_{M}$, then for every $\delta \in (\pi_M)^{-1}(\calO_h)$ we have $U_{\eta}^M(\g)\simeq U_{\delta}^M(\g)$.

\subsection{}
\label{subsec:Q}
In 1992, De Concini, Kac, Procesi formulated the following conjecture: 
\begin{conjecture}
\label{conj:DCKP}
If $V$ is an irreducible $U_{\eta}^{M}(\g)$-module with $\pi_M(\eta)=g$ then $\ell^{\frac{1}{2}\dim\calO_g}\mid\dim V$.
\end{conjecture}

 Since the map $\pi_M$ is compatible with the tensor product decomposition of $U^M_\ve(\g)$, verification of the conjecture can be reduced to the case of $\g$ simple.

\bigskip

Note that, since the diagram in \S \ref{subsec:DCKPmap} is commutative and since $\phi_{NM}$ is a central isogeny, \\
$\dim \calO_{\pi_M(\iota^*\eta)}^{G_M}=\dim \calO_{\pi_N(\eta)}^{G_N}$ for any $\eta\in \Spec(Z_0^N(\g))$.

The following  questions arise:

\begin{enumerate}
\item \label{Q1} Let $M\subset N$. Under which assumptions Conjecture \ref{conj:DCKP} for the lattice $M$ follows from or implies Conjecture \ref{conj:DCKP} for the lattice $N$? 
\item \label{Q2}Let $\eta\in \Spec(Z_0^M(\g))$.  Does there exist a $U_{\eta}^{M}(\g)$-module of dimension $\ell^{\frac{1}{2}\dim\calO_{\pi_M(\eta)}}$? Is it irreducible?
\item \label{Q3} Let $M\subset N$. Under which assumptions can we deduce an  answer to Question \eqref{Q2} for the lattice $M$ from the case of the lattice  $N$ or viceversa? 
\end{enumerate}

Question \eqref{Q2} is the quantum analogue of a problem posed by Humphreys, on representations of restricted Lie algebras.  We call a module $V$ whose dimension satisfies  
 an equality as in Question \eqref{Q2} a {\em small module} for  $U_{\eta}^{M}(\g)$. If the DCKP holds for $\pi_M(\eta)$, then any corresponding small module is irreducible. We show in \S\ref{sec:l_prime}
 that Question \eqref{Q2} does not always have an affirmative answer and we formulate necessary conditions under which an affirmative answer can be expected.
%a conjecture on when it is true. 

\subsection{}
We close this section by noticing that if $(\ell,|N/M|)=1$, the center $Z_{\ve}^{M}(\g)$ of $U_{\ve}^{M}(\g)$ behaves well with respect to inclusion. This fact will not be needed in the sequel.

\begin{lemma}
\label{lem:centers}
Let $M\subset N$ be such that  $(\ell,|N/M|)=1$. Then, $Z_{\ve}^{M}(\g)\subseteq Z_{\ve}^{N}(\g)$. In particular, this holds for every $M,N$ whenever $(\ell, |\Lambda/Q|)=1$.
\end{lemma}

\begin{proof}
The algebra $U_{\ve}^{N}(\g)$ is generated by $U_{\ve}^{M}(\g)$ and the elements $K_{\omega}$ with $\omega$ running through a set $N_{M}$ of representatives of $N/M$ in $N$. Therefore  $z\in Z_{\ve}^{M}(\g)$ lies in $Z_{\ve}^{N}(\g)$ if and only if $[K_\omega,z]=0$ for every $\omega\in N_M$. Let $z=\sum z_i$ be the expression of $z$ with respect to the PBW basis. %as a sum of PBW basis vectors. , up to nonzero constant.
By construction of the PBW basis,  $K_\omega z_i =c_{i\omega} z_iK_{\omega}$ for $c_{i\omega}$ some power of $\ve$. By linear independence of the $z_i$, for every $k\geq 0$ and every $\omega$ we have $[K_\omega^k,z]=0$ if and only if $[K_\omega^k,z_i]=0$ for every $i$, which happens if and only if $c_{i\omega}^k=1$ for every $i$. Since $K_\omega^{|N/M|}\in U_{\ve}^{M}(\g)$, for every $\omega$ and every $i$ we have $c_{i\omega}^{|N/M|}=1$, which under our hypotheses forces $c_{i\omega}=1$, whence the statement.
\end{proof}

The map $\xi_M$ in \S \ref{subsec:DCKPmap}  is the composition of the (surjective) restriction maps $\chi_M$ to $Z_{\ve}^{M}(\g)$ and  $\tau_M$ to $Z_0^{M}(\g)$, so in the the diagram
\begin{figure}[!h]
\centering
\begin{tikzpicture}
  \matrix (m) [matrix of math nodes, row sep=3em, column sep=3em]{
    |[name=UM]| \Spec (U_{\ve}^{M}(\g)) & |[name=ZeM]| \Spec (Z_{\ve}^{M}(\g))& |[name=Z0M]|\Spec ( Z_{0}^{M}(\g) )\\
    |[name=UN]| \Spec (U_{\ve}^{N}(\g)) & |[name=ZeN]| \Spec (Z_{\ve}^{N}(\g) )& |[name=Z0N]|\Spec (Z_{0}^{N}(\g))\\
    };
  \draw[->]
        (UM) edge node[auto] {$\chi_M$} (ZeM)
        (ZeM) edge node[auto] {$\tau_M$} (Z0M)

        (UN) edge node[auto] {$\chi_N$} (ZeN)
        (ZeN) edge node[auto] {$\tau_N$} (Z0N)

        (ZeN) edge node[auto] {$\iota^{\ast}$} (ZeM)
        (Z0N) edge node[auto] {$\iota^{\ast}$} (Z0M)
      
  ;
\end{tikzpicture}
\end{figure}

the first vertical arrow is surjective if $(\ell, |N/M|)=1$.% by Lemma \ref{lem:centers}. 

%------------------------------------------------------------------------
%
% RED QUANT ENV ALGS
%
%------------------------------------------------------------------------
\section{Reduced quantized enveloping algebras}\label{sec:reduced}

In order to deal with the questions from \S\ref{subsec:Q}, we compare reduced algebras corresponding to different lattices $M\subset N$. 

\subsection{}\label{subsec:red_basis}
Let $\eta_{M}\in \Spec (Z_0^M(\g))$ and $\eta_{N}\in \Spec (Z_0^N(\g))$ such that $\iota^* \eta_{N}=\eta_{M}$. It follows from the inclusion $Z_0^M(\g)\subset Z_0^N(\g)$ that there is a natural algebra morphism 
$f_{MN}\colon U_{\eta_{M}}^M(\g)\to U_{\eta_N}^N(\g)$. Indeed, we always have %$Z_{0}^{M}\subseteq Z_{0}^{N}\cap U_{\ve}^{M}$ so
$$
({\rm Ker\;}\eta_M )U_{\ve}^M(\g)\subseteq ({\rm Ker\,}\eta_N\cap Z_0^M(\g))U_{\ve}^M(\g)\subseteq ({\rm Ker\,}\eta_{N})U_{\ve}^N(\g).
$$

Note that, with respect to PBW bases corresponding to the same reduced decomposition of $w_0$, we have $f_{MN}({\bf F}^{A}K_{\lambda}{\bf E}^{C})={\bf F}^{A}f_{MN}(K_{\lambda}){\bf E}^{C}$ for any $\lambda\in M$. Hence the dimension of the image equals $\ell^{2N}\dim f_{MN}(U_{\eta}^{M}(\h))$. 

%$f_{MN}$ maps monomials of the PBW-basis of $U_{\eta_M}^{M}(\g)$ to the similar monomials in $U_{\eta_M}^{N}(\g)$ up to a nonzero scalar. The situation is as follows: we have two finite dimensional algebras $A_M$ and $A_N$, a subalgebra $E$ of both, $K_M$ a subalgebra of $A_M$ and $K_N$ a subalgebra of $A_N$ such that $A_M=E\otimes_{\CC}K_M\otimes_{\CC}E$ and $A_N=E\otimes_{\CC}K_N\otimes_{\CC}E$ as vector spaces. Moreover, the algebra morphism $f_{MN}$ equals $\id\otimes_{\CC}(f_{MN}|_{K_{M}})\otimes_{\CC}\id$, hence $f_{MN}(A_M)=E\otimes_{\CC} f_{MN}(K_M)\otimes_{\CC} E$. Therefore, in understanding $f_{MN}$ we may restrict to considering $f_{MN}|_{K_M}$.

For the purpose of analyzing $f_{MN}$, note that we are not bound to considering a specific basis $X$ of $M$. For each $\mu\in X$, $\eta(K_{\mu})\neq 0$ since $K_{\mu}$ is invertible. Fix $c_{\mu}^{\eta}$ an
$\ell$-th root of $\eta(K_{\mu})$ and denote $K_{\mu}/c_{\mu}^{\eta}$ by $K_{\mu}^{\eta}$. Clearly, replacing the $K_{\mu}$ in the monomials of the basis (paragraph \ref{subsec:DCKPmap}) with $K_{\mu}^{\eta}$ 
we still have a basis of $U_{\eta}^{M}(\h)$. The latter is isomorphic to the group algebra  $\CC[M/\ell M]$. Let  $f_{N}':\CC[N/\ell N]\rightarrow \CC[N/\ell N]\simeq U_{\eta}^{N}(\h)$  
be the invertible linear map defined by 
$f_{N}^{'}(K_{\mu}^{\eta_N})=(c_{\mu}^{\eta_M}/c_{\mu}^{\eta_N})K_{\mu}^{\eta_N}$. Then 
$k_{MN}=f_{N}'\circ f_{MN}$ restricts to the canonical group homomorphism $M/\ell M\rightarrow N/\ell N$. 

\begin{lemma}
\label{lemma:lattices}
Let $N$ be a lattice and $M\subseteq N$ a sublattice of finite index. The natural group homomorphism $k_{MN}:M/\ell M\rightarrow N/\ell N$ is an isomorphism if and only if $(|N/M|,\ell)=1$.
\end{lemma}
\begin{proof}
Note that  $\ker k_{MN}=(\ell N\cap M)/\ell M$. Assume that $(|N/M|,\ell)=1$. 
Since $k_{MN}$ is an endomorphism of the finite group $Z_{\ell}^{\rank N}$ it is enough to show that $k_{MN}$ is surjective. By assumption, there are $a,b\in\ZZ$ such that $a\ell+b|N/M|=1$ so, 
for any $x\in N$ we have $x=a\ell x+b|N/M|x\in \ell N+ M$. If instead there is a prime $p$ dividing $(|N/M|,\ell)$, then, there is $\mu\in N$ such that $\mu\not\in M$ and $p\mu\in M$. 
Then, $\ell\mu=\frac{\ell}{p} (p\mu)\in \ell N\cap M$ and $\ell\mu\not\in \ell M$ hence $k_{MN}$ is not injective.
\end{proof}

For our purposes we will have to consider lattices $M\subseteq N$ for which $(|N/M|,\ell)\neq 1$. As $\ell$ is odd, they occur only for Lie algebras with components of type $A_{m}$ or $E_6$. Recall that the simple roots are in Bourbaki ordering.

\begin{lemma}
\label{lem:AnE6}
If $\g$ is simple of type $A_n$ or $E_6$, then there is $\lambda_{\Lambda}\in\Lambda$ such that $\lambda_{\Lambda},\alpha_1,\dots,\alpha_{n-1}$ is a basis for $\Lambda$ and any lattice $Q\subseteq M\subseteq \Lambda$ 
is generated by $|\Lambda/M|\lambda_{\Lambda},\alpha_1,\dots,\alpha_{n-1}$.
\end{lemma}
\begin{proof} 
We use \cite[\S13.2]{humphreys1978algebras}. For $A_n$ we have $(n+1)\lambda_1=\sum_{i=1}^{n}(n-i+1)\alpha_i$, so $\Lambda=\la\lambda_1,Q\ra$ and since the coefficient 
of $\alpha_{n}$ is $1$, we have $\Lambda=\la \lambda_1,\alpha_1,\dots,\alpha_{n-1}\ra$ as claimed with $\lambda_{\Lambda}=\lambda_1$. For $E_6$ we choose 
$\lambda_{\Lambda}:=\lambda_3-\lambda_5=\frac{1}{3}(\alpha_1+2\alpha_3-2\alpha_5-\alpha_6)$. Since $\lambda_{\Lambda}\in\Lambda\setminus Q$ and $\Lambda/Q=Z_3$ we have 
$\Lambda=\la \lambda_{\Lambda},Q\ra$. As $-3\lambda_{\Lambda}\in \lambda_n+\la\alpha_1,\dots,\alpha_{n-1}\ra$ the claim follows also in this case. The last claim follows from the fact that $\Lambda/Q$ is cyclic.
\end{proof}

The following theorem relates different isogeny types for reduced quantized enveloping algebras. 
A result comparing different isogeny types for the infinite-dimensional algebras $U_{\ve}^{M}(\g)$ is described in \cite[\S5]{DCM}. 

\begin{theorem}
\label{thm:iso}
Let $Q\subseteq M\subseteq N\subseteq\Lambda$ with $M_i\subseteq N_i$ corresponding to the simple factors of $\g$. Then $U_{\eta}^{N}(\g)$ is a free $f_{MN}(U_{\iota^{\ast}(\eta)}^{M}(\g))$-module of rank $\prod_{i}(\ell,|N_i/M_i|)$. In particular $f_{MN}$ is an algebra isomorphism if and only if $(|N/M|,\ell)=1$.
%\begin{enumerate}[(a)]
%\item \label{thm:iso0} If $(\ell,|N/M|)=1$ then $f_{MN}$ is an algebra isomorphism.
%\item \label{thm:iso1} If $(\ell,|N/M|)=1$ and $\ll$ is a standard Levi then $f_{MN}$ restricts to an isomorphism of algebras $U_{\eta}^{M}(\ll)\rightarrow U_{\eta}^{N}(\ll)$.
%\item \label{thm:iso2} If $Q\subseteq M$ then $U_{\eta}^{N}(\g)$ is a free $f_{MN}(U_{\iota^{\ast}(\eta)}^{M}(\g))$-module of rank $\prod_{i}(\ell,|N_i/M_i|)$.
%\end{enumerate}
\end{theorem}

\begin{proof}
It is enough to prove the statement for $\g$ simple of rank $n$. Consider the group homomorphism $k_{MN}:M/\ell M\rightarrow N/\ell N$ and note that $\dim k_{MN}(\CC[M/\ell M])$
equals the order of the image of $N/\ell M\rightarrow M/\ell M$. By Lemma \ref{lemma:lattices}, when $k_{MN}$ is not an isomorphism, then $d:=(\ell, |N/M|)\neq 1$ and $\g$ is of type $A_n$ or $E_6$. 
Moreover in this case, by Lemma \ref{lem:AnE6}, $k_{MN}$ is the endomorphism $Z_{\ell}^{n}\rightarrow Z_{\ell}^{n}$ which restricts to the identity on the last $n-1$ $Z_\ell$-factors and on 
the first factor restricts to $x\mapsto |N/M|x$. Then, a basis of $f_{MN}(U_{\iota^{\ast}\eta}^{M}(\g))$ is given by  ${\bf F}^{A}K_{|\Lambda/M|\lambda_\Lambda}^{db}K_{\alpha_1}^{b_1}\cdots K_{\alpha_{n-1}}^{b_{n-1}}{\bf E}^{C}$, where 
$b\in {0,\ldots, \frac{\ell}{d}-1}$, $b_i\in {0,\ldots,{\ell}-1}$, and $A,C$ are as in \S \ref{subsec:red_basis}.
%
%The kernel of the restriction to the first factor consists of all elements of order dividing $|N/M|$, 
%i.e. all elements of order dividing $d:=(\ell,|N/M|)$. Hence, by Lemma \ref{lemma:lattices} in all cases $\dim f_{MN}(U_{\iota^{\ast}\eta}^{M}(\h))=\dim k_{MN}(\CC[M/\ell M])=\ell^{n}/d$. 
%It follows that $f_{MN}(U_{\iota^{\ast}\eta}^{M}(\g))$ is a subalgebra of $U_{\eta}^{N}(\g)$ of dimension $\ell^{\dim\g}/d$. In particular, if $d=1$ then $\dim f_{MN}(U_{\iota^{\ast}\eta}^{M}(\g))=\dim U_{\eta}^{N}(\g)$ and $f_{MN}$ is an isomorphism of algebras.
%
%If $d\neq 1$ and $\ell=d\ell'$, the image of the restriction of $k_{MN}$ to the first factor is represented by $d\{0,\dots,\ell'-1\}$ hence if $0<x<d$ then 
%$K_{\lambda_{N}}^{x}\notin f_{MN}(U_{\iota^{\ast}\eta}^{M}(\g))$. There are $d-1$ such elements and all normalize the lines through basis vectors of  $U_{\eta}^{N}(\g)$, hence $U_{\eta}^N(\g)$ is free of rank $d$.
\end{proof}

%------------------------------------------------------------------------
%
% lattices
%
%------------------------------------------------------------------------

%A direct consequence of the proof of the theorem is the following corollary.
%
%\begin{corollary}
%If $d=(\ell,|N/M|)\neq 1$, i.e. if $\g$ is simple of type $A_n$ or $E_6$, then $U_{\eta}^{N}(\g)$ is a free $f_{MN}(U_{\iota^{\ast}(\eta)}^{M}(\g))$ module with basis $\{K_{\lambda_{N}}^{i}\:0\leq i< d\}$ where $\lambda_M$ is as in Lemma \ref{lem:AnE6}.
%\end{corollary}
%comgio{Question: can it happen that the two algebras are isomorphic, even if $f_{MN}$ is not an iso? The counterexample below says that in that case they are not isomorphic, but what can we say in general?}

Theorem \ref{thm:iso} and the discussion in \S\ref{subsec:DCKPmap} give the following answer to Questions \eqref{Q1} and \eqref{Q3}. 

\begin{corollary}
\label{cor:l_prime}
Let $Q\subseteq M\subseteq N \subseteq \Lambda$ and let $\eta\in\Spec (Z_0^N(\g))$. If $(\ell,|N/M|)=1$ then
\begin{enumerate}[(a)]
\item the DCKP conjecture holds for $U_\eta^N(\g)$ if and only if it holds for $U_{\iota^{\ast}(\eta)}^M(\g)$; 
\item there exists a small $U_\eta^N(\g)$-module if and only if there exists a small $U_{\iota^{\ast}(\eta)}^M(\g)$-module.
%\item the quantum analogue of Humphreys conjecture holds  for $U_\eta^N(\g)$ if and only if it holds  for $U_{\iota^{\ast}(\eta)}^M(\g)$.
\end{enumerate}
\end{corollary}

\newcounter{tabla}\stepcounter{tabla}
\renewcommand{\thetabla}{\Roman{tabla}}

 Let $b(\g)$ be the maximum between the largest bad prime for $\g$ and the maximum $m$ such that $\Delta$ contains a subset of type $A_{m-1}$.  The values of $b(\g)$ for $\g$ simple are listed in \ref{tab:bg}
\begin{align}\label{tab:bg}\tag*{Table \thetabla}
\begin{tabular}{|c|c|c|c|c|c|c|c|c|}
\hline 
$A_n$  & $B_n$ & $C_n$ &  $D_n$  & $E_6$ & $E_7$ & $E_8$  & $F_4$ & $G_2$  \\
\hline
$n+1$  & $n$ & $n$ &  $n$  & $6$ & $7$ & $8$  & $3$ & $3$  \\
\hline
\end{tabular} 
 \end{align}
 
\begin{remark}
The strategy proposed in \cite{sev1}, \cite{sev2} aims at settling the DCKP-conjecture for $(\ell, n!)=1$ if $\g$ is of type $A_n$ and  $(\ell,b(\g)!)=1$ otherwise, in the case of $M=\Lambda$. So, if $(\ell, b(\g)!)=1$,  Corollary \ref{cor:l_prime} together with this result would imply the DCKP conjecture for every lattice $M$. 
\end{remark}

\section{Some positive answers to Question 2}
%\section{Reduction statements}% to Levi subgroups}
\label{sec:positive}

In this section we apply an inductive argument on the rank of $\g$ in order to give affirmative answers to Question \ref{Q2}, under certain coprimality assumptions on $\ell$. With notation explained in the sequel, 
there are two main parts in the argument: a reduction to $U_{\eta}^{N}(\ll)$ for some Levi subalgebra $\ll$ of $\g$ as in \cite{gio-almeria} and a further reduction to a subalgebra determined by $[\ll,\ll]$ for which the coprimality 
condition is needed. By parabolic induction, the problem of determining the existence of small modules is reduced to rigid orbits, see Remark \ref{rem:rigid}, hence the existence of small modules is settled for $\mathfrak{sl}_{n+1}$
when $(\ell, (n+1)!)=1$.

\subsection{}
\label{par:compatible}
For $\Pi\subset\Delta$, let  $Q_{\Pi}=\ZZ\Pi$, $\Phi_{\Pi}=Q_{\Pi}\cap\Phi$, $\Phi_{\Pi}^{+}=Q_{\Pi}\cap\Phi^{+}$. We denote the weight lattice of $\Phi_\Pi$ by $\Lambda_{\Pi}$ and the longest element of the corresponding 
Weyl group by $w_0^{\Pi}$. Let $\p$ be the associated standard parabolic subalgebra of $\g$ with standard Levi factor $\mathfrak{l}$.
For a lattice $N$ between $Q$ and $\Lambda$, if $P$ and $L$ are the connected subgroups of $G_N$ with $\Lie(P)=\p$ and $\Lie(L)=\ll$, then $P=L U_{P}$ for a connected unipotent subgroup $U_{P}\subseteq U$. We set $U_P^-:=w_0 U_P w_0^{-1}$.

Let $U_\ve^N(\mathfrak{p})$ be the subalgebra of $U_\ve^N(\mathfrak{g})$ generated by $F_{\alpha},K_{\gamma},E_{\beta}$ for $\alpha\in \Pi$, $\beta\in \Delta$, $\gamma\in N$, 
let $U_\ve^N(\mathfrak{l})$ be the subalgebra of  $U_\ve^N(\mathfrak{p})$ 
generated by $F_{\alpha},K_{\gamma},E_{\beta}$, $\alpha,\beta\in \Pi,\gamma\in N$.

When dealing with such subalgebras, we always assume that the reduced  decomposition of $w_0$ 
is such that the first $|\Phi_\Pi^+|$ terms form a reduced decomposition of $w_0^{\Pi}$. 
This way, the root vectors corresponding to roots in $\Phi_\Pi^+$ will be contained in 
$U_\ve^N(\mathfrak{l})$. 

For $\eta\in \spec(Z_0^N(\mathfrak{g}))$, let 
$$U_\eta^N(\mathfrak{p})=U_\ve^N(\mathfrak{p})/(({\rm Ker}\eta)U_\ve^N(\mathfrak{g})\cap U_\ve^N(\mathfrak{p})) \mbox{ and } U_\eta^N(\mathfrak{l})=U_\ve^N(\mathfrak{l})/(({\rm Ker}\eta)U_\ve^N(\mathfrak{g})\cap U_\ve^N(\mathfrak{l})).$$
By construction,  $U_\eta^N(\mathfrak{p})$ and $U_\eta^N(\mathfrak{l})$ are subalgebras of $U_\eta^N(\mathfrak{g})$ generated, respectively by  $F_{\alpha},K_{\gamma},E_{\beta}$ for $\alpha\in \Pi$,
$\beta\in \Delta$, $\gamma\in N$, and  $F_{\alpha},K_{\gamma},E_{\beta}$, $\alpha,\beta\in \Pi,\gamma\in N$. Note that $U_\eta^N(\mathfrak{l})$ depends only on the restriction of 
$\eta$ to the subalgebra of $Z_0^N(\g)$ generated by $F^\ell_{\alpha},K^\ell_{\gamma},E^\ell_{\beta}$, $\alpha,\beta\in \Phi^+_\Pi,\gamma\in N$.

%be the subalgebra of $U_\eta^N(\mathfrak{g})$ generated by $\calB_{\p}$ %. %We have an induction map
%\begin{equation}
%\label{p_to_g}
%\Spec(U_\eta^N(\pp))\to \Rep( U_\eta^N(\g))
%,\quad
%V\mapsto U_\eta^N(\g)\otimes_{U_\eta^N(\pp)}V.

%when $\Phi$ is simply laced; $d(\Pi')=2$ when $\Phi$ is double laced and $\Phi_{\Pi'}$ simply laced; $d(\Pi')=3$ when the type of $\g$ is $G_2$ and $\Pi$ contains only the long root.

% consider a simple factor $\tilde \g$ of $[\ll,\ll]$ which is of type $A_m$ and which corresponds to long roots $\tilde\Pi\subseteq\Pi$. From the presentation of the algebras we see that $U_{\ve}^{N}(\ll)$ 
%contains $U_{\ve^{d}}^{Q_{\Pi}}([\ll,\ll])$ as subalgebra where $d=2$ when $\g$ is not of type $G_2$ or $d=3$ when the type of $\g$ is $G_2$ and $\Pi$ contains only the long root. With the assumptions in the 
%introduction $(\ell,d)=1$ so $\ve^{d}$ is a primitive $\ell$-th root of $1$. Moreover, the restriction of $\pi_N$ to the subalgebra $U_{\ve^{d}}^{Q_{\Pi}}([\ll,\ll])$ is by definition $\pi_{Q_{\Pi}}$ \cite{DCKP}.

\subsection{}
\label{par:levi_to_parabolic}
%For $\eta\in \spec(Z_0^N(\mathfrak{g}))$, let $U_\eta^N(\mathfrak{p})$ be the subalgebra of $U_\eta^N(\mathfrak{g})$ generated by $\calB_{\p}$ %. %We have an induction map
%\begin{equation}
%\label{p_to_g}
%\Spec(U_\eta^N(\pp))\to \Rep( U_\eta^N(\g))
%,\quad
%V\mapsto U_\eta^N(\g)\otimes_{U_\eta^N(\pp)}V.
%\end{equation}
%and let $U_\eta^N(\mathfrak{l})$ be the subalgebra generated by $\calB_{\ll}$.
If $\eta$ is such that $\eta(E_\alpha^\ell)=0$ for every $\alpha\in\Phi^+\setminus \Phi_\Pi$ then, extending trivially the action of $\{E_\alpha:\alpha\in\Phi^+\setminus \Phi_\Pi\}$ induces a natural map 
\begin{equation}
\label{l_to_p}
\Rep(U_\eta^{N}(\ll))\rightarrow \Rep(U_\eta^N(\pp)).
\end{equation}
If $V\in \Spec( U_\eta^N(\pp))$, then $E_\alpha V=0$ for every $\alpha\in\Phi^+\setminus \Phi_\Pi$. 
 Indeed, if $I$ is the left ideal of $U_\eta^N(\pp)$ generated by $\{E_\alpha:\alpha\in\Phi^+\setminus \Phi_\Pi\}$ then $IV$ is a proper submodule of $V$. Therefore, the map \eqref{l_to_p} 
restricts to a bijection $\Spec(U_\eta^{N}(\ll))\rightarrow \Spec(U_\eta^N(\pp))$.

The composition of the map \eqref{l_to_p} with extension of scalars to $U_{\eta}^{N}(\g)$, i.e., with $V\mapsto U_\eta^N(\g)\otimes_{U_\eta^N(\pp)}V$, is the induction map defined in \cite[\S2.1]{gio-almeria}
$$
\Ind_{\ll}^{\g,\eta}:\Spec(U_\eta^{N}(\ll))\to \Rep( U_\eta^N(\g)).
$$

\subsection{}
\label{par:levi_to_g}
We make use of the generalization of Lusztig-Spaltenstein induction \cite{lusp} to arbitrary elements described in \cite{gio-espo}. 
If $\calO_x^L$ is a conjugacy class in a Levi subgroup $L$ of a parabolic subgroup $P$, with decomposition $P=U_P^-L$, then
$\Ind_{L}^{G}(\calO_{x}^{L})$ is the unique conjugacy class in $G$ intersecting $U_P^- x$ in a dense subset. A conjugacy class which is not induced from a class in any proper Levi subgroup is called rigid. 
For $g\in G_{N}$ the semisimple and unipotent factors in the Jordan decoposition are denoted by $g_s$ and $g_u$ respectively.

% and is called the \emph{Levi-envelope} of $C_{G}(g_s)^{\circ}$ \cite{..}.
%If $L$ and $P$ are as in \S\ref{par:compatible}

%Let $L\subseteq G_N$ be the Levi subgroup of the standard parabolic subgroup $P=L\ltimes U_{\p}$ corresponding to $\Pi$ and let ${\rm Ind}_L^{G_N}$ denote induction of conjugacy classes from $L$ to $G_{N}$. Then $g'\in L$ and if $\calO_v$
%intersects $U_{\p}^{-}$ in a dense set then $g=\Ind_{L}^{{G_N}}(g')$ by definition. 

For $\eta\in \spec(Z_0^N(\g))$, let  $\eta_{\ll}\in \spec(Z_0^N(\g))$ be defined as
$\eta_{\ll}(x)=\eta(x)$ for every $x\in Z_0^N(\g)\cap U_\ve^N(\ll)$ and $0$ elsewhere. By our choice of reduced decomposition of $w_0$, the ordering in $\Phi^+$ begins with the roots in $\Phi^+_\Pi$ and the ordering in $-\Phi^+$ ends with the roots in $-\Phi^+_\Pi$. By the identification in \S \ref{subsec:poisson} and the definition of $\pi_N$ in  \S \ref{subsec:DCKPmap} we see that $\pi_N(\eta_{\ll})\in L$. 

% From the discussion so far, the following lemma follows.
%\comgio{insert that $g\in B^-$ and correct part 3}
\begin{lemma}
\label{lem:g_if_l}
For $\eta\in\Spec(U_{\ve}^{N}(\g))$ and $\Pi\subseteq\Delta$ let $L$ be the standard Levi subgroup corresponding to $\Pi$, $\ll=\Lie(L)$. Assume $g=\pi_N(\eta)\in B_{N}^{-}$ and 
% be its Lie algebra. % and let $\eta_{\ll}$ be the restriction of $\eta$ to $U_{\ve}^N(\ll)$. 
%Suppose that $\calO_{g}^{G_N}=\Ind_{L}^{G_N}(\calO_{\pi_N(\eta_{\ll})}^{L})$. 
 that $\calO_{g}^{G_N}=\Ind_{L}^{G_N}(\calO_{g'}^{L})$ where  $g'=\pi_{N}(\eta_\ll)$. 

\begin{enumerate}[(a)]
%\item \label{lem:g_if_l0} Then $\calO_{g'}^{L}=\calO_{\pi_{N}(\eta_{\ll})}^{L}$.
\item \label{lem:g_if_l1} If there exists an $U_{\eta}^{N}(\ll)$-module of dimension $\ell^{\frac{1}{2}\dim \calO_{g'}^{L}}$ then there exists a small module for $U_{\eta}^{N}(\g)$.
%$U_{\eta}^{N}(\g)$-module of dimension $\ell^{\frac{1}{2}\dim \calO_{g}^{G_N}}$. 
\item \label{lem:g_if_l2} Assume in addition that $L=C_{G_N}(Z(C_{G}(g_s)^{\circ})^{\circ})$.
If there exists a $U_{\eta}^{N}(\g)$-module of dimension $\ell^{\frac{1}{2}\dim \calO_{g}^{G_N}}$ then there exists a  $U_{\eta}^{N}(\ll)$-module of dimension $\ell^{\frac{1}{2}\dim \calO_{g'}^{L}}$ and one is irreducible if and only if the other is so. 
%the reversed implication in \eqref{lem:g_if_l1} also holds.
\end{enumerate}
\end{lemma}

\begin{proof}
%By \cite{DCK} we may assume that $g\in U^{-}$ and in particular 
%Since $g\in B_{N}^{-}$, we may use the induction map from \S\ref{par:levi_to_parabolic}. By \S\ref{par:compatible} we may assume that the PBW basis of $U_{\ve}^{N}(\g)$ is compatible with $\ll$.
%Then $g=v\pi_{N}(\eta_{\ll})$ for some $v\in U_P^{-}$ and where $\pi_{N}(\eta_{\ll})\in L$. Hence, if $\calO_g=\Ind_{L}^{G_N}(\calO_{g'}^{L})$ then $\calO_{g'}^{L}=\calO_{\pi_N(\eta_{\ll})}^{L}$.
%
Note that $U_{\eta}^{N}(\ll)\simeq U_{\eta_{\ll}}^{N}(\ll)$. If $V$ is an $U_{\eta}^{N}(\ll)$-module of dimension $\ell^{\frac{1}{2}\dim \calO_{g'}^{L}}$ then, by \cite[(2.2)]{gio-almeria}
%considering the PBW basis of $U_{\ve}^{N}(\g)$ compatible with $\ll$,
$$
\dim\Ind_{\ll}^{\g,\eta}V=\ell^{\frac{1}{2}|\Phi-\Phi_{\Pi}|}\dim V=\ell^{\frac{1}{2}(|\Phi-\Phi_{\Pi}|+ \dim \calO_{g'}^{L})}.
$$
By \cite[Theorem 1.3]{lusp}, \cite[Proposition 4.6]{gio-espo}, $\dim \Ind_{L}^{G_N}(\calO_{g'}^{L})=|\Phi-\Phi_{\Pi}|+\dim\calO_{g'}^{L}$ which proves (a).
%
%For the last claim,
% we may assume that the Jordan decomposition of $g$ is such that 
%$g_s\in T_N$ 
%
%and $g_u\in U^{-}$ \cite{DCK}. %, i.e. we may assume that $C_{G_N}(g_s)^{\circ}=L$. 
The main theorem in \cite[\S8]{DCK} applies when $L=C_{G}(Z(C_{G}(g_s)^{\circ})^{\circ})$, 
the minimal Levi subgroup containing $C_{G}(g_s)^{\circ}$ \cite[\S 3.1]{lusztig}  and the claim (b) follows.
\end{proof}

\begin{remark}
If the DCKP conjecture holds for the lattice $N$ and the integer $\ell$, 
the $U_\eta^N(\g)$-modules considered in the lemma are always irreducible.
% and conclusion.
\end{remark}

\subsection{}
\label{par:tensor}
The inductive argument that we are aiming for is for quantized enveloping algebras of semisimple Lie algebras. In order to make
this possible we want to pass from $U_{\eta}^{N}(\ll)$ to a product of quantized enveloping algebra corresponding to the simple factors of $[\ll,\ll]$, as suggested, in a special case, in \cite[Remark 8.1]{DCK}.

Let $[\ll,\ll]=\ll_1\oplus\cdots \oplus\ll_r$ be the decomposition of $[\ll,\ll]$ 
in simple factors, with $\Pi=\Pi_1\sqcup \cdots \sqcup \Pi_r$.
%$\Phi_{\Pi}=\Phi_1\sqcup \cdots \sqcup \Phi_r$ and
%$Q_{\Pi}=Q_{\Pi_1}\oplus\cdots \oplus Q_{\Pi_r}$,
We set $\ve_i:=\ve_\alpha$ if $(\alpha|\alpha)>2$ for all $\alpha\in\Pi_i$ and $\ve_i:=\ve$ otherwise. Then
for the $r$-tuple ${\boldsymbol{\epsilon}}$ consisting of the $\ve_i$'s, the subalgebra $U_{\boldsymbol{\epsilon}}^{Q_{\Pi}}([\ll,\ll])$ of  $U_\ve^Q(\mathfrak{l})$ generated by 
$F_{\alpha},K_{\gamma},E_{\beta}$ for $\alpha,\beta,\pm\gamma \in \Pi$ is isomorphic to $\bigotimes_{i=1}^r U_{\ve_i}^{Q_{\Pi_i}}(\ll_i)$.

It follows from the construction that for our choice of a reduced decompostion of $w_0$ and $w_0^\Pi$, the root vectors of each  $U_{\ve_i}^{Q_{\Pi_i}}(\ll_i)$ are exactly the $E_\alpha,F_\beta$ for $\alpha,\beta\in\Phi_{\Pi_i}^+$. 
Thus, $Z_0^Q(\g)\cap U_{\boldsymbol{\epsilon}}^{Q_{\Pi}}([\ll,\ll])$ is the tensor product of the $\ell$-centers of the $U_{\ve_i}^{Q_{\Pi_i}}(\ll_i)$. 

Now, let $U_{\boldsymbol{\eta}}^{Q_{\Pi}}([\ll,\ll]):=U_{\boldsymbol{\epsilon}}^{Q_{\Pi}}([\ll,\ll])/(({\rm Ker}\,\eta\, U_\ve^{Q}(\g))\cap U_{\eta}^{Q_{\Pi}}([\ll,\ll]))$. Then, for $\eta_i$ the restriction of $\eta$ 
    to the $\ell$-center of  $U_{\ve_i}^{Q_{\Pi_i}}(\ll_i)$, this algebra is the product of restricted quantized enveloping algebras, for possibly different primitive $\ell$-th roots of unity:
\begin{equation}\label{eq:decomposition}
U_{\boldsymbol{\eta}}^{Q_{\Pi}}([\ll,\ll])\simeq \otimes_{i=1}^r U_{\ve_i}^{Q_{\Pi_i}}(\ll_i)/(\ker\eta_i)U_{\ve_i}^{Q_{\Pi_i}}(\ll_i).
\end{equation}
For the subset of simple roots $\Pi\subset \Delta$ associated to $\ll$, let $N_\perp^\Pi:=N\cap \Pi^\perp$, where $\Pi^\perp$ is the orthogonal subgroup of $\Pi$ in $\Lambda$ with respect to the natural pairing. Let $K_N^\Pi$ be the central subalgebra of  $U_\eta^N(\mathfrak{l})$ generated by $K_\mu$ for $\mu\in N^\Pi_\perp$.

%We denote by $A_\Pi$ the Cartan matrix corresponding to $\Phi_{\Pi}$. Recall that $\det(A_\Pi)=|\Lambda_\Pi/Q_\Pi|$. 
Identifying $\Pi$ and $\Delta\setminus \Pi$ with the set of corresponding 
indices parametrizing the simple roots, we see that a basis for $\Lambda_\perp^\Pi$ is given by $\{\lambda_i,\, i\not\in \Pi\}$. In general, 
$N_{\perp}^{\Pi}$ has rank $|\Delta-\Pi|$ and 
$K_{N}^{\Pi}\simeq{\CC}[Z_\ell^{|\Delta\setminus\Pi|}]$.

The following Lemma partially generalizes Theorem \ref{thm:iso}.

\begin{lemma}
\label{lem:tensor}
Let $Q\subseteq N\subseteq \Lambda$ and let $\ll$ be a standard Levi subalgebra corresponding to $\Pi\subseteq\Delta$. 
If $(\ell,|N/Q_{\Pi}\oplus N_{\perp}^{\Pi}|)=1$ then $U_{\eta}^{N}(\ll)\simeq U_{\boldsymbol{\eta}}^{Q_{\Pi}}([\ll,\ll])\otimes K_{N}^{\Pi}$.
%$U_{\eta}^{N}(\ll)\simeq U_{\eta}^{Q_{\Pi}}([\ll,\ll])\otimes K_{N}^{\Pi}$.
%In particular, if $(\ell,|\Lambda_{\Pi}/Q_{\Pi}|)=1$ then $U_{\eta}^{\Lambda}(\ll)\simeq U_{\eta}^{Q_{\Pi}}([\ll,\ll])\otimes K_{\Lambda}^{\Pi}$,
\end{lemma}
\begin{proof}
%We distinguish between the generators in $U_\eta^N(\mathfrak{l})$ and the generators in $U_\eta^{Q_\Pi}([\mathfrak{l},\mathfrak{l}])$ by labeling them with an upper index $N$ or $Q$. 
Consider the product map
$j\colon U_{\boldsymbol{\eta}}^{Q_{\Pi}}([\ll,\ll])\otimes K_{N}^{\Pi}\to U_\eta^N(\mathfrak{l})$.
% defined $j|_{U_\eta^{N}(\Pi)}=\id$ and $j(K^Q_{\mu})=K_\mu^N$ for $\mu\in N_\perp^{\Pi}$. 
%It is an algebra map because $K_{N}^{\Pi}$ is central in $U_\eta^N(\mathfrak{l})$.
%We distinguish between the generators in $U_\eta^N(\mathfrak{l})$ and the generators in $U_\eta^{Q_\Pi}([\mathfrak{l},\mathfrak{l}])$ by labeling them with an upper index $N$ or $Q$. Consider the map $j\colon U_\eta^{Q_\Pi}([\mathfrak{l},\mathfrak{l}])\otimes K_{N}^{\Pi}\to U_\eta^N(\mathfrak{l})$ defined by $j(E_\alpha^Q)=E^N_\alpha$, $j(F^Q_\alpha)=F^N_\alpha$, $j(K_\alpha^Q)=K_\alpha^N$ for $\alpha\in \Pi$ and $j(K^Q_{\mu})=K_\mu^N$ for $\mu\in N_\perp^{\Pi}$. It is an algebra map because $K_{N}^{\Pi}$ is central in $U_\eta^N(\mathfrak{l})$.
Let $M=Q_{\Pi}\oplus N_{\perp}^{\Pi}$ and $m=|N/M|$. For any $\mu\in N$, $m\mu=x+y$ with $x\in Q_{\Pi}$ and $y\in N_\perp^\Pi$. If $(\ell,m)=1$ then there are $b_1,b_2\in\ZZ$ such 
that $\mu=(b_1\ell+b_2m)\mu=b_1\ell\mu+b_2x+b_2y$. Hence
$$
K_{\mu}=\eta((K_{b_1\mu})^{\ell})K_{b_2x}K_{b_2y}\in j(U_{\boldsymbol\eta}^{Q_{\Pi}}([\ll,\ll])\otimes K_{N}^{\Pi})
$$
and $j$ is surjective. Using the PBW bases we see that $\dim U_{\boldsymbol{\eta}}^{Q_\Pi}([\mathfrak{l},\mathfrak{l}])\otimes K_{N}^{\Pi}=\dim U_\eta^N(\mathfrak{l})$, whence the statement.
%
%
%$$
%\dim U_{\boldsymbol{\eta}}^{Q_\Pi}([\mathfrak{l},\mathfrak{l}])\otimes K_{N}^{\Pi}\leq\ell^{\dim[\ll,\ll]+|\Delta-\Pi|}=\dim U_\eta^N(\mathfrak{l}).
%$$
%Therefore, if $j$ is surjective then the two algebras have the same dimension and $j$ is an isomorphism.
\end{proof}

We treat now the two special cases $N=\Lambda$ and $N=Q$.

\begin{lemma}Let  $\ll$ be a standard Levi subalgebra corresponding to $\Pi\subseteq\Delta$, let $A_{\Pi}$ be the Cartan matrix of $[\ll,\ll]$. Assume $(\ell, \det (A_\Pi))=1$. Then, 
\begin{equation}
%U_{\eta}^{\Lambda}(\ll)\simeq U_{\eta}^{N}(\Pi)\otimes K_{\Lambda}^{\Pi}\mbox{ and }U_{\eta}^{Q}(\ll)\simeq U_{\eta}^{N}(\Pi)\otimes K_{Q}^{\Pi}.
U_{\eta}^{\Lambda}(\ll)\simeq U_{\boldsymbol{\eta}}^{Q_{\Pi}}([\ll,\ll])\otimes K_{\Lambda}^{\Pi}\mbox{ and }U_{\eta}^{Q}(\ll)\simeq U_{\boldsymbol{\eta}}^{Q_{\Pi}}([\ll,\ll])\otimes K_{Q}^{\Pi}.
\end{equation}
\end{lemma}
\begin{proof}
Recall that $\det (A_\Pi)=|\Lambda_\Pi/Q_{\Pi}|$. In order to show that $j$ as in Lemma \ref{lem:tensor} is an isomorphism when $N=\Lambda$, 
we observe that for two lattices $M\subseteq N$ of rank $n$ and $A_{NM}$ the matrix expressing $n$ generators of $M$ in terms of 
$n$ generators of $N$ we have $|N/M|=\det A_{NM}$. Let $M=Q_{\Pi}\oplus\Lambda^{\Pi}_{\perp}$ and $N=\Lambda$. Since a basis of $\Lambda^{\Pi}_{\perp}$ is a subset of the fundamental weights, 
after a suitable reordering of the vectors, we may assume that the matrix $A_{NM}$ is the block diagonal matrix ${\rm diag}(A_{\Pi}, I_{|\Delta\setminus\Pi|})$. In particular $\det A_{NM}=\det A_{\Pi}$.

We consider now the case $N=Q$. Let $V={\mathbb Q}\Delta$. We rearrange the basis $\Delta$ in such a way that the first elements are the simple roots in $\Pi$.
Let $DA$ be the symmetrized Cartan matrix of $\Phi$, let $D_\Pi$ be the square submatrix of $D$ 
consisting of its first $|\Pi|$ rows and columns, and let 
$A'=(D_\Pi A_\Pi\quad C)$ be the $|\Pi|\times |\Delta|$ submatrix of $DA$ consisting of the first $|\Pi|$ rows. 
The orthogonal subspace $V_\perp^\Pi$ to $\Pi$ in $V$ is given by those vectors whose coordinate columns $X$ with respect to the reordered basis $\Delta$ are solutions
to $A'X=0$. Mulitplying by $(D_\Pi A_\Pi)^{-1}$ on the left we have 
$(I_{|\Pi|}\quad (D_\Pi A_\Pi)^{-1}C)X=0$. For $B:=(D_\Pi A_\Pi)^{-1}C$, 
a ${\mathbb Q}$-basis for $V_\perp^\Pi$ is then given by the vectors of the form
$\beta_i:=-\sum_{j=1}^{|\Pi|}b_{ji}\alpha_j+\alpha_{|\Pi|+i}$, for $i=1,\ldots,|\Delta\setminus\Pi|$. Let $M:=\langle\Pi,\beta_j,\,j=1,\ldots,|\Delta\setminus\Pi|\rangle$.
Then, we have $Q_{\Pi}\oplus Q_{\perp}^{\Pi}\subset Q\subset M$ and 
$|Q/Q_{\Pi}\oplus Q_{\perp}^{\Pi}|$ divides $|M/Q_{\Pi}\oplus Q_{\perp}^{\Pi}|$. Now, the exponent of $M/Q_{\Pi}\oplus Q_{\perp}^{\Pi}$ divides 
$\det (D_\Pi A_\Pi)$ because $\det (D_\Pi A_\Pi)\beta_j\in Q_{\Pi}\oplus Q_{\perp}^{\Pi}$ for every $j$. Since $(\ell, \det (A_\Pi))=1$ and $(\ell, d_i)=1$, we have  $(\ell,|M/Q_{\Pi}\oplus Q_{\perp}^{\Pi}|)=1$, whence
$(\ell,|Q/Q_{\Pi}\oplus Q_{\perp}^{\Pi}|)=1$. Lemma \ref{lem:tensor} applies.
%it suffices to show that the exponent of $\Lambda/Q_{\Pi}\oplus\Lambda_{\perp}^{\Pi}$ divides $(\det A_{\Pi})$, 
%i.e. for every $\mu\in \Lambda$, we have $ (\det A_\Pi)\mu \in Q_\Pi+\Lambda_\perp^{\Pi}$.
%Indeed, for $j\in \Pi$, let $x_{j}\in{\mathbb Q}$ be the set of solutions of the linear system 
%\begin{equation}
%$$
%\sum_{j\in\Pi}x_j\langle \alpha_j,\alpha_i\rangle =\langle \mu,\alpha_i\rangle
%$$
%\end{equation}
%Then for $y_j:=(\det A_\Pi )x_j\in{\mathbb Z}$, for $j\in \Pi$ and for every $i\in \Pi$ we have 
%$$\langle (\det A_\Pi)\mu-\sum_{j\in\Pi}y_j \alpha_j, \alpha_i\rangle=0$$
%so $(\det A_\Pi)\mu-\sum_{j\in\Pi}y_j \alpha_j \in \Lambda_\perp^{\Pi}$.
\end{proof}

Let $Q\subseteq N\subseteq\Lambda$ be such that $(\ell,|N/Q_\Pi\oplus N_\perp^\Pi|)=1$ and let $\eta\in \Spec(Z_0^N(\g))$. An immediate consequence of the above lemma is the following equality of sets
$$
\{\dim V~|~V\in \Spec(U_\eta^N(\ll))\}=\{\dim V~|~V\in \Spec( U_{\boldsymbol{\eta}}^{Q_{\Pi}}([\ll,\ll]))\}.$$
%\{\dim V~|~V\in \Spec(U_\eta^N(\ll))\}=\{\dim V~|~V\in \Spec(U_\eta^{Q_\Pi}([\ll,\ll]))\}.
%
%$$

%\comgio{Until here}
%Assume in addition that $\pi_N(\eta)\in B_N^-$. Then, for any fixed character $\zeta$ of $K_N^{\Pi}$, we have an induction map 
%$$
%{\Ind}_{\Pi}^{N,\chi}:\Spec (U_{\eta}^{N}(\Pi))\to {\rm Rep}(U_\eta^{N}(\g))
%{\Ind}_{[\ll,\ll]}^{N,\zeta}:{\rm Rep} (U_{\boldsymbol{\eta}}^{Q_\Pi}([\ll,\ll]))\to {\rm Rep}(U_\eta^{N}(\g)).
%{\rm Ind}_{\Pi}^{M,\chi}:\Spec (U_\eta^{Q_\Pi}([\l,\l]))\to {\rm Rep}(U_\eta^{N}(\g))
%$$
%as in \cite{gio-almeria}. 
%It is defined as follows: for $V\in \Spec(U_{\boldsymbol{\eta}}^{Q_{\Pi}}([\ll,\ll]))$, extend scalars to $U_\eta^N(\ll)$ via $\zeta$ using Lemma \ref{lem:tensor} and compose with the induction map in \S\ref{par:levi_to_parabolic}.
%In other words, $\Ind_{\Pi}^{N,\chi}$ is the composition of the following maps
%$$
%\Spec (U_{\eta_\Pi}^{Q_\Pi}([\ll,\ll]))
%\to
%\Spec (U_{\eta_\Pi}^{N}(\ll))
%\to
%\Spec (U_{\eta_\Pi}^{N}(\p))
%\to
%{\Rep}(U_\eta^{N}(\g))
%$$
%Since the PBW-bases of $U_\eta^{Q_\Pi}([\ll,\ll])$ and $U_\eta^{N}(\g)$ are compatible we see that $\dim\Ind_{\Pi}^{N,\chi}V=\ell^{\frac{1}{2}|\Phi-\Phi_{\Pi}|}\dim V$. Note also that the last two maps do not depend on $\chi$.

\subsection{}
We are now in a position to prove the main statement. Recall that a unipotent conjugacy class $\calO$ is called Richardson if it is induced from the trivial class in some Levi subgroup of a parabolic subgroup. 
Note that for unipotent conjugacy classes this property of $\calO$ does not depend on the isogeny type of $G_N$. As $U^-$ does not depend on the isogeny type, 
if $\eta$ is such that $\pi_N(\eta)\in U^-$, 
then $\pi_M(\iota^*\eta)=\pi_N(\eta)$ for every $M\subset N$ and we  simply write 
$\pi_M(\eta)$.

\begin{lemma}\label{lem:second}
 Let $Q\subseteq N\subseteq\Lambda$, $\eta\in \Spec (Z_0^N(\g))$ and $g=\pi_{N}(\eta)\in B_N^-$. Assume that $(\ell,|N/Q_\Pi\oplus N_\perp^\Pi|)=1$ and that 
$L=C_{G_N}(g_s)^{\circ}$ is the standard Levi subgroup of a standard parabolic subgroup associated with $\Pi\subset \Delta$.
Let $g_u\in [L,L]=L_1\cdots L_r$ decompose as a product $g_u=h_1\cdots h_r$, so that $\pi_{Q_{\Pi_i}}(\eta_i)=h_i$.
If each $U_{\ve_i}^{Q_{\Pi_i}}(\ll_i)/(\ker\eta_i)U_{\ve_i}^{Q_{\Pi_i}}(\ll_i)$ has a module of dimension $\ell^{\frac{1}{2}\dim\calO_{h_i}^{L_i}}$, then 
$U_\eta^N(\g)$ has a small module. 
\end{lemma}
\begin{proof}
Under these assumptions, $U_{\boldsymbol{\eta}}^{Q_{\Pi}}([\ll,\ll])$ has a module of dimension $\ell^{\frac{1}{2}\dim\calO_{g_u}^{L}}$. Therefore, by Lemma \ref{lem:tensor}, $U_\eta^N(\mathfrak{l})$ has a module of the same dimension. 
Lemma \ref{lem:g_if_l} (b) applies.
\end{proof}

\begin{lemma}
\label{lem:first}
Let $\eta\in \Spec (Z_0^N(\g))$ and $g=\pi_{N}(\eta)\in U^-$. If $\Oc_{g}^{G_N}$ is Richardson, then $U_\eta^N(\g)$ has a small module. In particular this hold if $\g=\mathfrak{sl}_{n+1}$ and $\pi_{N}(\eta)\in U^-$. 
%Let $Q\subseteq N\subseteq\Lambda$, $\eta\in \Spec (Z_0^N(\g))$ and $g=\pi_{N}(\eta)\in B_N^-$. 
%Suppose that $L=C_{G_N}(g_s)^{\circ}$ is the standard Levi subgroup of a standard parabolic subgroup associated with $\Pi\subset \Delta$ and that 
%$\Oc_{g_u}^{L}$ is Richardson. Assume in addition that $(\ell,|N/Q_\Pi\oplus N_\perp^\Pi|)=1$.
%$(\ell, |\Lambda_\Pi/Q_\Pi|)=1$.
%Then $U_\eta^N(\g)$ has a small module. 
\end{lemma}
%having Jordan decomposition such that $g_s\in T_N$ and $g_u\in U^-$.
\begin{proof}
%We have $\Oc_g^G={\rm Ind}_{L}^G(\Oc_{g}^L)$  \cite[Proposition 2.6]{gio-almeria}. By Lemma \ref{lem:g_if_l} it suffices to show that $U_{\eta}^{N}(\ll)$ has a module of dimension $\ell^{\frac{1}{2}\dim\calO_{g_u}^L}$. 
%By Lemma \ref{lem:tensor} and the assumption on $\ell$, this is equivalent to showing that $U_{\eta}^{N}(\Pi)$ has a module of the same dimension. Observe that $g_s\in Z(L)$, that 
%$g_u\in [L,L]=L_1\cdots L_r$ and that $g_u$ decomposes as a product $g_u=g_1\cdots g_r$ with 
%$d_u=\prod_{i=1}^r\ell^{\frac{1}{2}\dim\calO_{g_i}^{L_i}}$ and $\pi_{Q_{\Pi_i}}(\eta_i)=g_i$.
%It suffices to prove that each $U_{\ve_i}^{Q_{\Pi_i}}(\ll_i)/(\ker\eta_i)U_{\ve_i}^{Q_{\Pi_i}}(\ll_i)$ have a module of dimension $\ell^{\frac{1}{2}\dim\calO_{g_i}^{L_i}}$. 
%As in \S\ref{par:compatible} and \S\ref{par:levi_to_parabolic} we may assume that the PBW basis of $U_{\ve}(\g)$ is compatible with $\ll$ and that $\pi_{Q_{\Pi}}(\eta_{\ll})=g_u=\prod_{x\in\calC(\Pi)}\pi_{x}(\eta_{x})$ for some ordering of the components $\calC(\Pi)$. It suffices to treat the case where $[\ll,\ll]$ is simple. Then $U_{\eta}^{N}(\Pi)$ is the reduced enveloping algebra of $U_{\eta}^{Q_{\Pi}}([\ll,\ll])$ of $U_{\ve^{d(\Pi)}}^{Q_{\Pi}}([\ll,\ll])$. Hence, it suffices to treat the case where $g_s=1$ independent of the primitive root $\ve$ of one.
%So we assume that $[L,L]$ is simple.
By assumption we have $g\in\Ind_{L}^{G_N}(\Oc_1^{L})$, for some standard Levi subgroup $L$ of a parabolic subgroup of $G$, associated with $\Pi$. 
By Lemma \ref{lem:g_if_l} (a) it is enough to show that $U_{\eta_{\ll}}^N(\mathfrak{l})=U_{1}^N(\mathfrak{l})$ has a $1$-dimensional module. Being a Hopf algebra, the counit gives a small module. The last statement follows because all unipotent conjugacy classes in type $A_n$ are Richardson.
%Now, the small quantum group $U_{\boldsymbol{1}}^{Q_{\Pi'}}([\mathfrak{l}',\mathfrak{l}'])$ is a Hopf algebra, hence it  has a $1$-dimensional module, namely the counit. By Lemma \ref{lem:g_if_l} applied to $L'\subset L$ the statement follows.
%and \S\ref{par:tensor} we need to prove that $U_{\eta_L}^{Q_\Pi}([\mathfrak{l},\mathfrak{l}])$, with $\eta_L$ the restriction of $\eta$ to $Z_0^{Q_\Pi}([\mathfrak{l},\mathfrak{l}])$, has a small module. By assumption we have $\pi_{Q_{\Pi}}(\eta_L)=g_u\in\Ind_{L'}^L(\Oc_1^{L'})$, for some standard Levi subgroup $L'$ of a parabolic subgroup of $G_{Q_{\Pi}}$, associated with $\Pi'\subset\Pi$. Now, the small quantum group $U_{1}^{Q_{\Pi'}}([\mathfrak{l}',\mathfrak{l}'])$ is a Hopf algebra, hence it  has a $1$-dimensional module, namely the counit. By Lemma \ref{lem:g_if_l} applied to $L'\subset G_{Q_{\Pi}}$ we have the statement.
% As $g_s$ is central in $L$, we have $\pi_{Q_{\Pi}}(\eta_L)=g_u\in\Ind_{L'}^L(\Oc_1^{L'})$, for some standard Levi subgroup $L'$ of a parabolic subgroup of $[L,L]$, associated with $\Pi'\subset\Pi$. 
\end{proof}

We are ready to state our result on small modules for $\g=\mathfrak{sl}_{n+1}({\mathbb C})$. 

\begin{theorem}
\label{cor:sln}
If $(\ell, (n+1)!)=1$ then $U_\eta^M(\mathfrak{sl}_{n+1})$ has a small module for every $\eta\in\Spec(Z_0^M(\mathfrak{sl}_{n+1}))$ and every lattice $M$.
\end{theorem}
\begin{proof}
By Lemma \ref{thm:iso}, it suffices to consider the lattice $M=Q$.  The connected centralizer $L$ of any semisimple element in $G={\rm PSL}_{n+1}({\mathbb C})$
is a Levi subgroup of a parabolic subgroup of $G_N$ corresponding to some $\Pi\subseteq\Delta$ and every unipotent class in $[L,L]$ is Richardson. We apply Lemmas \ref{lem:second} and \ref{lem:first}.
\end{proof}

The condition on $\ell$ in the theorem  is in accordance to the one given by Friedlander and Parshall in their proof of 
the existence of small modules for $\mathfrak{sl}_{n+1}$ in the modular case \cite[Theorem 5.1]{fri-par}.

\begin{remark}\label{rem:rigid}Lemma \ref{lem:first} and Theorem \ref{cor:sln} are special cases of the following pattern, which is similar to \cite[Theorem 5.1]{fri-par}.
By transitivity of induction, combining Lemma \ref{lem:tensor} with Lemma \ref{lem:g_if_l} (a) and the product decomposition in \eqref{eq:decomposition}, we see that if there exists a small module for 
all rigid conjugacy classes in all simple factors of Levi subgroups of $G$, then there is a small module for all conjugacy classes in $G$.  
\end{remark}

\section{The case $(\ell, |N/M|)\neq1$}

\label{sec:l_prime}

In this section we deal with Questions \eqref{Q1}, \eqref{Q2} and \eqref{Q3} when $\g$ is simple of type $A_n$ or $E_6$,  $M\subsetneq N$ and $d:=(\ell, |N/M|)\neq1$. 
With respect to the dimension of $\g$, the smallest case to consider is type $A_2$ when $3\mid \ell$. From a different perspective, with respect 
to $\dim \calO_{\pi_{N}(\eta)}$ the smallest cases are the central $\ell$-characters. Type $A_2$ is considered in \S \ref{par:A2} 
and we describe precisely when there exists a small module for a central $\ell$-character and all $\g$ in \S\ref{par:irred_central}.
%The latter is also proved in a more general setup in \S\ref{sec:cyclic}.
We observe that if $\pi_N(\eta)=1$ then $U_1^N(\g)$ is a Hopf algebra, so it  always has a small module, namely the one given by the counit. For this reason we will deal only with $\eta$ such that $\pi_N(\eta)\neq1$.

\subsection{}
\label{par:central}
An $\ell$-character $\eta\in Z_0^{N}(\g)$ is called \emph{central} if $\pi_N(\eta)\in Z(G_{N})$. 
For such characters we have $\eta(E_{\alpha}^{\ell})=\eta(F_{\alpha}^{\ell})=0$ for all $\alpha\in\Phi^{+}$. 
Moreover, $\pi_N(\eta)\in Z(G_N)$ if and only $\alpha(\pi_N(\eta))=1$ for every $\alpha\in \Delta$, i.e., if and only if $\eta(K_{\alpha}^{2\ell})=1$ for all $\alpha\in Q$.
%$\pi_{N}(\eta)$ is in the kernel of the isogeny $\phi_{NQ}$ restricted to $T_{N}$ so $\phi_{NQ}(\pi_{N}(\eta))=1$ if and only if $\alpha(\phi_{NQ}(\pi_{N}(\eta)))=1$ for all $\alpha\in\Delta$. Considering the commutative diagram in \S\ref{subsec:DCKPmap}, this is equivalent to $\alpha(\iota^{\ast}(\pi_{Q}(\eta)))=1$ for all $\alpha$ $\Leftrightarrow$ $\eta(K_{\iota(\alpha)}^{2\ell})=1$ for all $\alpha$.
%As in \S\ref{subsec:DCKPmap} this isogeny is the comorphism $\iota^{\ast}$ of the inclusion map $\CC[T_Q]=\CC[Q]\subseteq\CC[N]=\CC[T_N]$ so $\phi_{NQ}(\pi_{N}(\eta))=1$ if and only if $\alpha(\phi_{NQ}(\pi_{N}(\eta)))=1$ for all $\alpha\in\Delta$. Since the diagram in \S\ref{subsec:DCKPmap} commutes this is equivalent to $\alpha(\iota^{\ast}(\pi_{Q}(\eta)))=1$ for all $\alpha$ $\Leftrightarrow$ $\eta(K_{\iota(\alpha)}^2)=1$ for all $\alpha$. %which in turn is equivalent to $\iota(\alpha)(\pi_Q(\eta))=1\Leftrightarrow $ for all $\nu\in N$. 
In our setting, by Lemma \ref{lem:AnE6}, a central $\ell$-character
$\eta$ is uniquely determined by $\eta(K_{\lambda_N}^{2\ell})=\omega$ a $|N/Q|$-th root of one. By \S\ref{subsec:automorphisms} it suffices to treat the case $\eta(K_{\lambda_N}^\ell)=\omega$.

\subsection{}
\label{par:trick}
If it exists, a small module $V$ for a central $\ell$-character $\eta$ has dimension $1$. We recall from \cite[\S 9.1]{DCP} that in this case $E_{\alpha}V=F_{\alpha}V=0$ for any $\alpha$ and
\begin{equation}
E_\alpha F_\alpha-F_\alpha E_\alpha=\frac{K_{\alpha}-K_{-\alpha}}{\ve_{\alpha}-\ve_{\alpha}^{-1}}
%\Rightarrow
%(K_{\alpha}-K_{\alpha}^{-1})V=0
\Rightarrow
K_{\alpha}^{2}=\id_V.
\end{equation}
%\begin{remark}\label{rem:trick}
In general, consider a finite-dimensional $U_{\ve}^M(\g)$-module $W$. 
%\begin{enumerate}
%\item
If $E_\alpha W=0$ or $F_\alpha W=0$ for some $\alpha\in \Phi$, then $K_\alpha^2=\id_{W}$. %This follows from the relation $E_\alpha F_\alpha-F_\alpha E_\alpha=\frac{K_\alpha-K_\alpha^{-1}}{\ve-\ve^{-1}}$. 
%\item
Conversely, if $K_\alpha^2=\id_{W}$, then for any $\beta\in \Phi$ such $(\alpha|\beta)\not\equiv 0\mod \ell$ we have $F_\beta W=E_\beta W=0$.
Indeed, as $K_\alpha$ is diagonalizable, it is enough to show that  $E_\beta$ and $F_\beta$ act trivially on each of its eigenspaces.
However, these operators map any $\pm 1$-eigenvector to the $\pm\ve^{\pm(\alpha,\beta)}$-eigenspace. The latter is trivial because $\pm\ve^{\pm(\alpha,\beta)} \neq1$.
%\end{enumerate} 

\subsection{}
\label{par:A2}
%Question \eqref{Q2} has a negative answer, when $d\neq 1$.
We explore here the case where $\g=\mathfrak{sl}_3(\CC)$ and $3\mid \ell$. The example below shows that Question \eqref{Q2} from \S \ref{subsec:Q} has a negative answer in general. 

%\begin{example}
%\label{rem:counter}
%Let $3|\ell$, $\g=\mathfrak{sl}_3({\mathbb C})$, $U_\ve^\Lambda(\g)$. In this case $G_N=SL_3({\mathbb C})$ and we look for small modules for  $z\in Z(SL_3({\mathbb C}))$, that is, $1$-dimensional modules lying over $z$.
%As $K_\lambda E_i K_{-\lambda}=\ve^{\lambda(\alpha_i)}E_i$, and similarly for $F_i$, and since $K_\lambda$ is invertible, then any $1$-dimensional representation $V={\mathbb C}$ of $U_\ve^\Lambda(\g)$satisfies $E_i.1=F_i.1=0$. Since $E_iF_i-F_iE_i=\frac{K_{\alpha_i}-K_{-\alpha_i}}{(\ve-\ve^{-1})}$, then $K_\lambda^2.1=1$ for every $\lambda\in Q$. 

%As $K_\lambda E_i K_{-\lambda}=\ve^{\lambda(\alpha_i)}E_i$, and similarly for $F_i$, and since $K_\lambda$ is invertible, then any $1$-dimensional representation $V={\mathbb C}$ of $U_\ve^\Lambda(\g)$ satisfies $E_i.1=F_i.1=0$.

By \S\ref{par:trick} the possible $1$-dimensional representations of $U_\ve^\Lambda(\mathfrak{sl}_3)$ are in bijection with $\Hom(\Lambda/2Q,{\mathbb C})$. Let $V=\CC$ be a $1$-dimensional module affording the central $\ell$-character $\eta$.
As $3\lambda_i\in Q$, we have $K_{\lambda_i}^6.1=1$ on $V$ and since $3\mid\ell$, necessarily $K_{\lambda_i}^{2\ell}.1=1$ on $V$. The image of $V$ through the map defined in \cite[\S4.3]{DCKP} is the element $t=\pi_{\Lambda}(\eta)$ in 
$T_{\Lambda}$ such that $\lambda_i(t)=\eta(K_{\lambda_i}^{2\ell})$ for every $i$. But $\eta(K_{\lambda_i}^{2\ell})=1$ on $V$ hence $t=\pi_\Lambda(\eta)=1$. Therefore, if $\pi_{\Lambda}(\eta)\in Z(SL_3({\mathbb C}))$ is 
of order $3$, there does not exist a small module for $U_{\eta}^{\Lambda}(\mathfrak{sl}_3)$. 
%\end{example}

We show now that the minimal dimension  of an irreducible $U_{\ve}^\Lambda(\mathfrak{sl}_3)$-module with central $\ell$-character $\eta$ with $\pi_\Lambda(\eta)\neq1$ is at least $3$.  
%We have $\Lambda=\la\lambda,\alpha\ra$ where $\lambda=\lambda_{\alpha}$, $(\lambda_\alpha,\alpha)=1$. %so $\alpha=2\lambda_\alpha-\lambda_{\beta}$. 
%Also $[x]=\frac{\ve^{x}-\ve^{-x}}{\ve-\ve^{-1}}$.
%$U_{\ve}^{\Lambda}(\mathfrak{sl}_3)$ is generated by $E_{\alpha}$, $E_{\beta}$, $F_{\alpha}$, $F_{\beta}$, $K_{\lambda}$ and $K_{\alpha}$ subject to relations as in \S\ref{sec:notation}. 
Let $w_0$ have reduced expression $s_{\beta}s_{\alpha}s_{\beta}$. 
%Then 
%$E_2:=E_{\beta}$, $E_{12}:=T_{\beta}(E_{\alpha})$, $E_{1}:=T_{\beta}T_{\alpha}(E_{\beta})=E_{\alpha}$ are the root vectors. 
From \cite[\S5]{lusztig-rootof1},  $E_{\beta+\alpha}=-E_{\beta}E_{\alpha}+\ve^{-1}E_{\alpha}E_{\beta}$.

%$$
%E_{12}=E_{\beta+\alpha}=T_{\beta}(E_{\alpha})=
%-E_{\beta}E_{\alpha}
%+\ve^{-1}E_{\alpha}E_{\beta}.
%$$
%$$
%\ve\lambda E_{12}v
%=K_{\alpha}E_{12}v
%=K_{\alpha}(-E_{2}w+\ve^{-1}E_{1}E_{2}v)
%=-\ve^{-1}(\ve^2\lambda) E_{2}w+\lambda E_{1}E_{2}v
%$$
By the discussion in \S\ref{par:trick} it suffices to show that if $W$ is a $U_{\eta}^\Lambda(\mathfrak{sl}_3)$-module of dimension $2$ then $E_\gamma. W=0$ for some $\gamma\in\Phi^{+}$ because then $\eta(K_\alpha^{2\ell})=1$ 
for every $\alpha\in Q$ and the previous argument applies.
% as follows. Assume there is an irreducible $U_{\eta}^\Lambda(\g)$-module $W$ of dimension $2$. There are two possibilities: 
%\begin{enumerate}[1.)]
%\item
%1.)
%$E_\alpha. W=0$ for every $\alpha\in Q$. Hence, $K_\alpha^2=\id_W$ for every $\alpha$, so as in \S\ref{par:trick} we have $K_{\lambda_i}^6.1=1$ for all $\lambda_i$.
%\item
There exists a basis $\{v,w\}$ of $W$ consisting of weight vectors for $U^\Lambda_\eta(\h)$. If $E_\alpha.W\neq 0$ then we may assume that $w= E_\alpha v\neq 0$.  Moreover, since $E_\alpha$ is nilpotent, 
$E_\alpha^2.W=0$. If $K_{\alpha}v=\lambda v$,
then $K_{\alpha}w=\ve^2\lambda w$. As $K_{\alpha}E_{\alpha+\beta}v=\ve\lambda E_{\alpha+\beta}v$, 
since $\ve\lambda\not\in\{\lambda,\ve^2\lambda\}$ we have $0=E_{\alpha+\beta}v=(-E_{\beta}E_{\alpha}+\ve^{-1}E_{\alpha}E_{\beta})v$. Hence $E_{\beta}w=\ve^{-1}E_{\alpha}E_{\beta}v$. Now $K_{\alpha}E_{\beta}w=\ve\lambda 
E_{\beta}w$ so as before $E_{\beta}w=0$
and $E_{\alpha+\beta}w=(-E_{\beta}E_{\alpha}+\ve^{-1}E_{\alpha}E_{\beta})w=-E_{\beta}E_{\alpha}w=0$ where the last equality follows from $E_\alpha^{2}v=E_\alpha w=0$. Hence $E_\gamma.W=0$ for $\gamma=\alpha+\beta$.
%By the discussion in \S\ref{par:trick} it follows that $K_{\alpha+\beta}=\id_{W}$ a contradiction.
%\end{enumerate}

\subsubsection{}
We look at the case $\ell=3$ and we show that there is an irreducible representation of $U^\Lambda_\varepsilon(\mathfrak{sl}_3)$ with central $\ell$-character $\pi_\Lambda(\eta)\neq1$ and dimension $3$.
By the above discussion, the minimal dimension of an irreducible $U^\Lambda_\eta(\mathfrak{sl}_3)$-module  is therefore exactly $3$.

Let $z\in{\mathbb C}$ be such that $z^3=\varepsilon^2$. The map $\rho\colon U^\Lambda_\eta(\mathfrak{sl}_3)\to  {\rm Mat}_3({\mathbb C})$ given on generators by
\begin{equation*}
E_{\alpha}\mapsto
\begin{bmatrix}
0 & 0 & 0
\\
0 & 0 & 1
\\
0 & 0 & 0
\end{bmatrix}
,\quad
E_{\beta}\mapsto
\begin{bmatrix}
0 & 0 & 0
\\
0 & 0 & 0
\\
1 & 0 & 0
\end{bmatrix}
,\quad
F_{\alpha}\mapsto
\begin{bmatrix}
0 & 0 & 0
\\
0 & 0 & 0
\\
0 & 1 & 0
\end{bmatrix}
,\quad
F_{\beta}\mapsto
\begin{bmatrix}
0 & 0 & 1
\\
0 & 0 & 0
\\
0 & 0 & 0
\end{bmatrix}
\end{equation*}
and
$K_{\lambda_1}\mapsto \operatorname{diag}(z, z^{-2},z)$, $K_{\alpha}\mapsto \operatorname{diag}(1,\varepsilon,\varepsilon^2)$ is a representation.
It is irreducible because $\rho$ is surjective. Its $\ell$-character $\eta$ is central, with $\eta(K_{\lambda_1}^6)=\varepsilon$. Since $U^\Lambda_\eta(\mathfrak{sl}_3)$ is a Hopf algebra we can construct
the dual representation. It is again irreducible and has central $\ell$-character $\eta^*$ with $\eta^*(K_{\lambda_1}^6)=\varepsilon^2$.

\subsection{}

The observations on central $\ell$-characters in the case of $\mathfrak{sl}_3$ can be generalized to arbitrary simple $\g$.  
\label{par:irred_central}

%For $\eta\in \Spec(Z_0^M(\g))$ such that $\pi_M(\eta)=z\in Z(G_M)$, the DCKP-conjecture is trivially true since $\ell^{\frac{1}{2}\dim\calO_{\eta}}=1$. We want to know when this bound is attained by some module.
% 
\begin{proposition}
\label{prop:central}
Let $\g$ be a simple Lie algebra, let $M$ be a lattice with $m:=|M/Q|$ and $d:=(\ell,|M/Q|)$. Let $\eta\in \Spec(Z_0^M(\g))$ be 
such that $\pi_M(\eta)=z\in Z(G_M)$. Then, there exists a $1$-dimensional representation of $U_\eta^M(\g)$ if and only if the order of $z$ divides $\frac{m}{d}$. 
\end{proposition}
\begin{proof}If $d=1$, Theorem \ref{thm:iso} yields $U_\eta^M(\g)\simeq U_1^Q(\g)$
%which is a Hopf algebra, so the counit will do. 
so we assume $d\neq 1$. Here $\g$ is either of type $E_6$ or $A_n$. We set $\ell':=\frac{\ell}{d}$ and $m':=\frac{m}{d}$.
From Lemma \ref{lem:AnE6}, $U_\ve^M(\h)$ has a ${\mathbb Z}$-basis given by $\lambda_M$,  $\alpha_1,\,\ldots,\,\alpha_{n-1}$, where the order of $\lambda_M Q$ in $M/Q$ is exactly $m$, and 
$m\lambda_M$,  $\alpha_1,\,\ldots,\,\alpha_{n-1}$ is a ${\mathbb Z}$-basis for $Q$. 
%Every $1$-dimensional representation $V={\mathbb C}v$ of $U_\ve^M(\g)$ satisfies $E_\alpha.v=F_\alpha.v=0$ for every $\alpha\in\Phi^+$ and $K_{\alpha_j}^2.v=v$ for every $j$.  Thus, 
If $V$ is a $1$-dimensional $U_\eta^M(\g)$-module, by \S \ref{par:trick}
we have $K_{\alpha_j}^{2}=\id_V$ for every $j$ and  $\lambda_M(z)=\eta(K_{\lambda_M}^{2\ell})$ so 
$\lambda_M(z^{m'})=\eta(K_{m\lambda_M}^{2\ell'})=1$, whence $z^{m'}=1$. Conversely, if $z^{m'}=1$, we define the following action of $U_\ve^M(\h)$ on $V={\mathbb C}v$: we set
$K_{\alpha_j}.v=v$  for every $j\neq n$ and for $a,b\in{\mathbb Z}$  and $\xi\in{\mathbb C}$ satisfying $d=a\ell+bm$ and $\xi^{2d}=\lambda_M(z)^a$, we set $K_{\lambda_M}.v=\xi v$.  
Then, $K_{m\lambda_M}^{2}.v=\xi^{2dm'}.v=\lambda_M(z^{m'})^{2a}.v=v$, so  $K^2_{\alpha_n}.v=v$. Thus, setting $E_\alpha.v=F_\alpha.v=0$  for every $\alpha\in\Phi^+$ gives a well-defined representation of $U_\ve^M(\g)$. As 
$E_\alpha^\ell=F_\alpha^\ell=0$ for every $\alpha$, we have $\chi_M(V)\in \Spec({\mathbb C}[K_{\mu}^{\pm\ell},\mu\in M])$. Moreover, $K_{\alpha_j}^{2\ell}=1$ for every $j$, so $\tau_M\chi_M(V)\in Z(G_M)$.
Finally, $K_{\lambda_M}^{2\ell}=\xi^{2\ell}=\xi^{2d\ell'}=\lambda_M(z)^{a\ell'}=\lambda_M(z)^{1-bm'}=\lambda_M(z)$. Hence, $V$ is a $1$-dimensional representation of $U_\eta^M(\g)$.
\end{proof}

Motivated  by the positive results from Section \ref{sec:positive} and the above discussion, we formulate a quantum analogue of Humphreys conjecture.

\begin{conjecture}\label{conj:small}
Let $\g$ be a Lie algebra with root system $\Phi$, let $M$ be a lattice satisfying $Q\subset M\subset \Lambda$, and let $\ell$ be such that $(\ell, b(\g)!)=1$. 
Then, for every  $\eta\in \spec(Z_0^M(\g))$  with  $\pi_M(\eta)\in \calO$ there exists an irreducible $U_{\eta}^{M}(\g)$-module $V$ such that $\dim V=\ell^{\frac{1}{2}\dim\calO}$.
\end{conjecture}

Conjecture \ref{conj:small} holds for  $\mathfrak{sl}_{n+1}$, Theorem \ref{cor:sln}. Evidence for this conjecture is also given by Lemmas \ref{lem:second} and \ref{lem:first}. 
Remark \ref{rem:rigid} shows that it is enough to consider rigid orbits.
The bound $b(\g)$ could be relaxed for specific $\eta$'s, 
see Proposition \ref{prop:central}, but seems to be necessary 
if we wish to have a general statement for $\g$.

%\subsection{Acknowledgements}
%This work is partially supported by Grant CPDA125818/12 of Padova University and MIUR PRIN $2012KNL88Y\_004$.
%I.I.S. was supported by ...

\end{document}